\documentclass[twoside, 10pt, letter]{article}

\usepackage{amsmath, amscd, amsfonts, amssymb, amsthm, latexsym, url}
\usepackage{graphicx}
\usepackage{thmtools, thm-restate}
\usepackage{todonotes}

\usepackage[left=1.3in, right=1.3in, top=1.3in, bottom=1.3in, includefoot, headheight=13.6pt]{geometry}
\input{xy}
\xyoption{all}
\objectmargin={2mm}

\usepackage{fancyhdr}

\usepackage{sectsty}
\sectionfont{\Large\sc\centering}
\subsectionfont{\sc}
\partfont{\huge\centering\sc}

\newtheorem{theorem}{Theorem}[section]
\newtheorem{lemma}[theorem]{Lemma}
\newtheorem{corollary}[theorem]{Corollary}
\newtheorem{proposition}[theorem]{Proposition}

\theoremstyle{definition}

\newtheorem{remark}[theorem]{Remark}
\newtheorem{example}[theorem]{Example}

\newcommand{\xysquare}[8]{
\[\xymatrix{
#1 \ar@{#5}[r] \ar@{#6}[d] & #2 \ar@{#7}[d]\\
#3 \ar@{#8}[r] & #4
}\]
}

\newcommand{\al}{\alpha}
\newcommand{\bb}{\mathbb}

\newcommand{\bor}{\partial}

\newcommand{\comment}[1]{}

\newcommand{\ep}{\varepsilon}

\newcommand{\isoto}{\stackrel{\simeq}{\to}}

\newcommand{\mult}[1]{#1^{\!\times}}

\newcommand{\onto}{\twoheadrightarrow}
\newcommand{\op}{\operatorname}
\newcommand{\pid}[1]{(#1)}
\newcommand{\res}{\overline}
\newcommand{\roi}{\mathcal{O}}

\newcommand{\sub}[1]{{\mbox{\scriptsize #1}}}

\newcommand{\To}{\longrightarrow}

\newcommand{\xto}{\xrightarrow}

\newcommand{\did}[1]{\langle #1\rangle}

\renewcommand{\cal}{\mathcal}
\renewcommand{\hat}{\widehat}
\renewcommand{\frak}{\mathfrak}
\newcommand{\indlim}{\varinjlim}
\renewcommand{\tilde}{\widetilde}

\renewcommand{\ker}{\operatorname{Ker}}
\renewcommand{\projlim}{\varprojlim}

\DeclareMathOperator{\Ann}{Ann}

\DeclareMathOperator{\Frac}{Frac}

\DeclareMathOperator{\Hom}{Hom}

\DeclareMathOperator{\Spec}{Spec}

\DeclareMathOperator*{\projlimf}{``\varprojlim\!''}
\DeclareMathOperator*{\holim}{\operatorname*{holim}}

\newcommand{\comp}{{\hat{\phantom{o}}}}

\begin{document}
\title{$K$-theory of one-dimensional rings via pro-excision}

\author{\sc Matthew Morrow\footnote{University of Chicago, supported by a Simons Postdoctoral Fellowship.}}

\date{}



\maketitle
\begin{abstract}
This paper studies `pro-excision' for the $K$-theory of one-dimensional, usually semi-local, rings and its various applications. In particular, we prove Geller's conjecture for equal characteristic rings over a perfect field of finite characteristic, give results towards Geller's conjecture in mixed characteristic, and we establish various finiteness results for the $K$-groups of singularities, covering both orders in number fields and singular curves over finite fields.


\end{abstract}

\section*{Introduction}
In the first section of this paper we will show that if $A$ is a one-dimensional, Noetherian, reduced, semi-local ring for which the normalisation morphism $A\to\tilde A$ is finite, then there is a long exact, Mayer--Vietoris, `pro-excision' sequence of pro abelian groups \[\cdots\to K_n(A)\to\projlimf_r K_n(A/\frak m^r)\oplus K_n(\tilde A)\to\projlimf_r K_n(\tilde A/\frak M^r)\to\cdots\tag{pro-MV},\] where $\frak m,\frak M$ denote the Jacobson radicals of $A,\tilde A$ respectively. There is also a similar sequence for the relative $K$-groups. Here $\projlimf_{r}$ denotes a pro abelian group, i.e.~a formal inverse system of groups, sometimes denoted, e.g., $\{K_n(A/\frak m^r)\}_r$.

Before discussing the main results of this paper, we explain the source of (pro-MV). It has been known at least since work by R.~Swan \cite[Thm.~3.1]{Swan1971} that $K$-theory fails to satisfy excision; i.e., if $A\to B$ is a morphism of rings and $I$ is an ideal of $A$ mapped isomorphically to an ideal of $B$, then $K_n(A,I)\to K_n(B,I)$ need not be an isomorphism. Having fixed $I$ as a non-unital algebra, A.~Suslin \cite{Suslin1995} showed, by building on earlier work of himself and M.~Wodzicki \cite{Suslin1992}, that $I$ satisfies excision for {\em all} such morphisms $A\to B$ if and only if $I$ is {\em homologically unital}, in Wodzicki's sense that $\op{Tor}^{\bb Z\ltimes I}_*(\bb Z,\bb Z)=0$ for $*>0$. Unfortunately, this is not commonly satisfied for rings of algebraic geometry. A recent trend has therefore been to consider instead the problem of `pro-excision': i.e.,\ When is the map $\projlimf_r K_n(A,I^r)\to\projlimf_rK_n(B,I^r)$ an isomorphism? For example, if $A$ is a Noetherian $\bb Q$-algebra then these maps are isomorphisms by a recent theorem of the author \cite[Thm.~0.1]{Morrow_Birelative}. Moreover, T.~Geisser and L.~Hesselholt \cite{GeisserHesselholt2006, GeisserHesselholt2011} have established a pro version of the Suslin--Wodzicki condition. In section \ref{section_pro_excision} we use Geisser--Hesselholt's results to show that if $I$ is the conductor ideal of a one-dimensional, Noetherian, reduced ring for which the normalisation map is finite, then it satisfies pro-excision, thereby resulting in long exact, Mayer--Vietoris, pro-excision sequences such as (pro-MV) above.

Such sequences have immediate global applications: If $X$ is a proper, reduced curve over a finite field, then pro-excision implies that $K_n(X)\to K_n(\tilde X)$ has finite kernel and cokernel for $n\ge 1$, whence $K_n(X)$ is finite by G.~Harder \cite{Harder1977} and C.~Soul\'e \cite{Soule1984}. In other words, Harder--Soul\'e's finiteness result extends to singular curves. The arithmetic analogue, which also follows from pro-exicision, is that if $A\subseteq \roi_F$ is an order in the ring of integers of a number field $F$, then $K_n(A)$ is finitely generated and of the same rank as $K_n(\roi_F)$; these ranks are of course known thanks to A.~Borel \cite{Borel1974}. The proofs of these results are postponed until section \ref{subsection_global} with the other material on rings with finite residue fields.

However, the major theme of this paper is of a local nature, showing that such pro-excision sequences often break into short exact sequences and studying the many interesting consequences, especially to Geller's conjecture in the finite residue characteristic case. In particular, sections \ref{subsection_finite_char} -- \ref{subsection_main_results} are devoted to the proof of the following key theorem, which can be interpreted as an analogue for singular rings of the Gersten conjecture:

\begin{theorem}\label{theorem_intro_1}
Let $A$ be a one-dimensional, Noetherian, reduced semi-local ring containing a field such that $A\to \tilde A$ is a finite morphism; let $\frak m$ and $\frak M$ denote the Jacobson radicals of $A$ and $\tilde A$. Then the relative version of (pro-MV) breaks into short exact sequences of pro abelian groups, yielding \[0 \to K_n(A,\frak m)\to\projlimf_r K_n(A/\frak m^r,\frak m/\frak m^r)\oplus K_n(\tilde A,\frak M)\to\projlimf_r K_n(\tilde A/\frak M^r,\frak M/\frak M^r)\to0\] for all $n\ge 0$. If $\tilde A\to\tilde A/\frak M$ splits (e.g., $A$ complete), then the non-relative sequence (pro-MV) also splits into short exact sequences: \[0\to K_n(A)\to\projlimf_r K_n(A/\frak m^r)\oplus K_n(\tilde A)\to\projlimf_r K_n(\tilde A/\frak M^r)\to0.\]
\end{theorem}

The special case of the theorem when $A$ is essentially of finite type over a field of characteristic zero and $n=2$ was proved by A.~Krishna \cite[Thm.~3.6]{Krishna2005}; the $n=2$ assumption was removed in the author's earlier work \cite{Morrow_Singular_Gersten} on this subject, but the characteristic and finite type assumptions remained. Therefore the theorem extends the results in \cite{Morrow_Singular_Gersten} on $K$-theoretic ad\`eles to all one-dimensional, Noetherian, reduced schemes over a field for which the normalisation map is finite. I suspect that the theorem is true without the equal characteristic assumption; section \ref{subsection_p_adic_orders} gives partial results in mixed characteristic.

Informally, the theorem states that the contributions to the $K$-theory of $A$ coming from its singularities can be entirely captured using the $K$-theory of the quotients $A/\frak m^r$, for $r\gg0$. For example, we use the theorem to prove the following in section \ref{subsection_KH}, where $KH$ denotes C.~Weibel's homotopy invariant $K$-theory:

\begin{theorem}\label{theorem_intro_1a}
Let $A$ be as in theorem \ref{theorem_intro_1}, and assume that $A\to A/\frak m$ splits. Then, for each $n\ge0$, the kernel of $K_n(A)\to KH_n(A)$ embeds into $K_n(A/\frak m^r)$ for $r\gg0$.
\end{theorem}

Moreover, precisely because (pro-MV) and the first theorem describe the singular contribution to the $K$-groups, they have important applications to Geller's conjecture, which we interpret as the following, although this is not exactly what S.~Geller asked in 1986 \cite{Geller1986}:
\begin{quote}
``Let $A$ be a one-dimensional, Noetherian, reduced local ring, and suppose that $K_2(A)\to K_2(F)$ is injective, where $F=\Frac A$ is the total quotient ring of $A$. Then $A$ is necessarily regular.''
\end{quote}
Apart from the seminormal, equal characteristic case, there has been no progress until now on the conjecture when $A$ has finite residue characteristic. We prove the following in section \ref{subsubsection_Geller}:

\begin{theorem}\label{theorem_intro_2}
Geller's conjecture is true if $A$ is an $\bb F_p$-algebra with perfect residue field for which $A\to \tilde A$ is a finite morphism.
\end{theorem}

Meanwhile, in section \ref{subsection_rel_to_HC} we offer the following interesting alternative to Geller's conjecture in characteristic zero:

\begin{theorem}\label{theorem_intro_2a}
If $A$ is as in Geller's conjecture, is essentially of finite type over a characteristic zero field, and is not regular, then $K_n(A)\to K_n(F)$ is not injective for some $n\ge 3$.
\end{theorem}

Next we turn our attention to mixed characteristic rings; section \ref{subsection_p_adic_orders} analyses consequences of the sequence (pro-MV) for reduced $\bb Z_p$-algebras which are finitely generated and torsion-free as $\bb Z_p$-modules, e.g.\ $\bb Z_p+p\bb Z_q$. If $A$ is such a ring then $\tilde A$ is a finite product of rings of integers of local fields, whose $K$-groups are largely understood thanks to L.~Hesselholt and I.~Madsen's proof of the Quillen--Lichtenbaum conjecture \cite[Thm.~A]{Hesselholt2003}. In particular, this implies that the following theorem holds for rings of integers of local fields; extending it to $A$ relies on the observation that pro-excision implies that $K_n(A)\to K_n(\tilde A)$ has finite kernel and cokernel for $n\ge1$.

\begin{theorem}\label{theorem_intro_3}
Let $A$ be a reduced $\bb Z_p$-algebra which is finitely generated and torsion-free as a $\bb Z_p$-module. Then
\[K_n(A)=\begin{cases}
\mbox{divisible $\bb Z_{(p)}$-module }\oplus\mbox{ finite $p$-group} & n\ge 2\mbox{ even,}\\
\mbox{torsion-free $\bb Z_{(p)}$-module }\oplus\mbox{ finite group} & n\ge 1\mbox{ odd,}
\end{cases}\]
\end{theorem}

It follows from group theory that the groups appearing in these direct sum decompositions are determined, up to isomorphism, by $K_n(A)$. For example, if $n$ is even then the finite $p$-group appearing in the theorem is necessarily the quotient of $K_n(A)$ by its maximal divisible subgroup; in fact, if $n$ is even then the standard short exact sequence \[0\to\op{Ext}^1_{\bb Z}(\bb Q/\bb Z,K_n(A))\to K_n(A;\hat{\bb Z})\to\Hom_{\bb Z}(\bb Q/\bb Z,K_{n-1}(A))\to 0\] implies, in conjunction with the previous theorem, that the finite $p$-group is precisely the $K$-group with $\hat{\bb Z}$ coefficients $K_n(A;\hat{\bb Z})$. A similar structural description of $K_n(A;\hat{\bb Z})$ when $n$ is odd is also given in corollary \ref{corollary_profinite_K_groups_of_p_adic_order}, and these results are used to prove the analogue of theorem \ref{theorem_intro_1} in odd degrees for $A$. In even degree we can reduce the problem to understanding the torsion in the even degree $K$-groups; this is well understood for $K_2$, but not in general, resulting in the following:

\begin{theorem}\label{theorem_intro_4}
Let $A$ be a reduced $\bb Z_p$-algebra which is finitely generated and torsion-free as a $\bb Z_p$-module. Then $K_2(A;\hat{\bb Z})$ equals the finite $p$-group alluded to in theorem \ref{theorem_intro_3} (and similarly for $\tilde A$), and there is a short exact, Mayer--Vietoris sequence \[0\to K_2(A)\to K_2(A;\hat{\bb Z})\oplus K_2(\tilde A)\to K_2(\tilde A;\hat{\bb Z})\to 0\]
\end{theorem}

Finally, using similar methods as in our proof of Geller's conjecture in equal finite characteristic, this theorem yields the first results towards Geller's conjecture in mixed characteristic. Our methods do not see the element $p$, so rather than establishing regularity under the conditions of the conjecture, we instead bound the embedding dimension of $A$ by $2$:

\begin{theorem}\label{theorem_intro_5}
Let $A$ be a one-dimensional, Noetherian, reduced local ring of mixed characteristic with finite residue field of characteristic $p>2$, and such that $A\to\tilde A$ is finite. Suppose that at least one of the following is true:
\begin{enumerate}
\item $\Frac\hat A$ contains no non-trivial $p$-power roots of unity; or
\item $A$ is seminormal, and $\hat A$ is not a certain exceptional case (see theorem \ref{theorem_geller_in_mixed_char} for details); or
\item $\tilde A$ is local and all $p$-power roots of unity in $\Frac\hat A$ belong to $\hat A$.
\end{enumerate}
If the map $K_2(A)\to K_2(\Frac A)$ is injective then $\op{embdim}A\le 2$.
\end{theorem}

In the remainder of this introduction we describe our methods and the ingredients of the proofs of the above results, beginning with theorem \ref{theorem_intro_1}. It is sufficient, using the relative version of (pro-MV), to show that \[K_n(\tilde A,\frak M)\to\projlimf_rK_n(\tilde A/\frak M^r,\frak M/\frak M^r)\] is surjective for all $n\ge 1$. Thus it is enough to show that the codomain is entirely symbolic and so, since $\tilde A/\frak M^r$ is a finite product of truncated polynomial rings, this reduces the theorem to proving that \[\projlimf_rK_n(k[t]/\pid{t^r},\pid t)\] is entirely symbolic for any field $k$. When $k$ has characteristic zero this is proved directly in section \ref{subsection_char_zero} by filtration arguments in cyclic homology after applying the following case of the Goodwillie isomorphism \cite{Goodwillie1986}: $K_n(k[t]/\pid{t^r},\pid t)\cong HC_{n-1}^{\bb Q}(k[t]/\pid{t^r},\pid t)$. The proof is philosophically similar in finite characteristic, in that we apply Hesselholt and Madsen's \cite{Hesselholt2001, Hesselholt2008} description of the $K$-theory of truncated polynomial rings in finite characteristic using topological cyclic homology via the McCarthy isomorphism \cite{McCarthy1997}.

Sections \ref{subsection_KH} -- \ref{subsection_rel_to_HC} cover the applications of (pro-MV) and the first main theorem; they may be read independently:

Firstly, theorem \ref{theorem_intro_1a} follows in a straightforward way from theorem \ref{theorem_intro_1} by relating the $KH$-theory of $A$ with the $K$-theory of $\tilde A$.

Next, (pro-MV) reduces theorem \ref{theorem_intro_2} to checking that \[\projlimf_rK_2(A/\frak m^r,\frak m/\frak m^r)\to\projlimf_rK_2(\tilde A/\frak M^r,\frak M/\frak M^r)\] is injective if and only if $A$ is regular. It can be shown using perfectness of the residue field that the codomain of this map vanishes, so it becomes enough to construct a non-zero element of $K_2(A/\frak m^r,\frak m/\frak m^r)$ for some $r$ under the assumption that $A$ is singular; this is easily achieved using Dennis--Stein symbols. We emphasise that the key to the proof is that (pro-MV) reduces the problem to the $K$-theory of the quotients $A/\frak m^r$ and $\tilde A/\frak M^r$. Krishna  \cite{Krishna2005} studied Geller's conjecture in a similar way in characteristic zero, introducing an `Artinian Geller's Conjecture' which mimicked G.~Corti\~nas, S.~Geller, and C.~Weibel's \cite{Weibel1998} Artinian analogue of Berger's conjecture on the torsion in differential forms of curve singularities.

Finally, theorem \ref{theorem_intro_2a} is proved by establishing an intermediate result which is of interest in its own right: The maps \[K_n(A,\frak m)\to K_n(\tilde A,\frak M),\quad\quad HC_{n-1}(A,\frak m)\to HC_{n-1}(\tilde A,\frak M)\] have isomorphic kernels (and cokernels). This is proved by establishing an analogue of theorem \ref{theorem_intro_1} for cyclic homology and then appealing to the Goodwillie isomorphism. For example, the resulting isomorphism of the kernels is uniquely determined by the commutativity of the diagram
\[\xymatrix@R=5mm{
\ker(K_n(A,\frak m)\to K_n(\tilde A,\frak M))\ar@{^(->}[d]\ar[r]^\cong& \ker(HC_{n-1}(A,\frak m)\to HC_{n-1}(\tilde A,\frak M))\ar@{^(->}[d]\\
K_n(A,\frak m) \ar[d] & HC_{n-1}(A,\frak m)\ar[d]\\
\projlimf_r K_n(A/\frak m^r,\frak m/\frak m^r)\ar[r]^\cong &\projlimf_r HC_{n-1}(A/\frak m^r,\frak m/\frak m^r)
}\]
where the bottom arrow is the Goodwillie isomorphism; the isomorphism of the cokernels is determined in a similar way. These isomorphisms reduce theorem \ref{theorem_intro_2a} to the same claim for cyclic homology, which can be deduced from the `Hochschild homology criterion for smoothness' \cite{Avramov1992}.

Now we say something about our proofs for mixed characteristic rings with finite residue fields. If $A$ is a reduced $\bb Z_p$-algebra which is finitely generated and torsion-free as a $\bb Z_p$-module, then pro-excision implies that the kernel and cokernel of $K_n(A)\to K_n(\tilde A)$ are finite; this reduces theorem \ref{theorem_intro_3} to the normal case, where it is known. The weaker mixed characteristic analogue of theorem \ref{theorem_intro_1} is deduced by explicitly examining the structural description offered by that theorem. We obtain the strongest result when $n=2$, namely theorem \ref{theorem_intro_4}, because the divisible summand of $K_2(\tilde A)$ is known to be torsion-free; the lack of this knowledge for the other even degree $K$-groups is what prevents us from fully proving theorem \ref{theorem_intro_1} in the mixed characteristic setting.

Theorem \ref{theorem_intro_5} can then be proved in a similar way as theorem \ref{theorem_intro_2}: theorem \ref{theorem_intro_4} reduces Geller's conjecture to an Artinian analogue. This can often be directly tackled at the level of $K_2(A/\frak m^2)$, which can be described in a classical style using differential forms and exterior powers of the cotangent space.

\subsection*{Acknowledgements}
I am grateful to C.~Weibel and T.~Geisser for interesting and helpful conversations during the Algebraic $K$-theory and Arithmetic meeting in Bedlewo, Poland, during July 2012. I also thank the anonymous referee who, as well as making many comments which improved the paper, identified certain inaccuracies and omissions in the original manuscript.

I greatly appreciate the Simons Foundation's support via a postdoctoral fellowship.

\subsection*{Notation, conventions, etc.}
Every ring in this paper is commutative, though we stress that the first two theorems of section \ref{section_pro_excision} remain valid in the associative, non-commutative case. Every ring is moreover unital, with the strict exception of certain instances in section \ref{section_pro_excision} where we write ``non-unital ring''.

Given a reduced ring $A$, we denote its total quotient ring by $\Frac A$, and its integral closure in $\Frac A$ by $\tilde A$. We often require $\tilde A$ to be finitely generated as an $A$-module, i.e. that $A\to \tilde A$ is finite (the term ``Mori ring'' is in the literature but we will not use it); this is true if $A$ is excellent \cite[7.8.3]{EGA_IV_II}. A ring is said to be normal if and only if it is reduced and integrally closed in its total quotient ring, i.e.~$A=\tilde A$.

Since most of this paper concerns local rings, $K_0$ is typically uninteresting and therefore we do not replace $K$-theory by its non-connective completion; in particular, our long exact Mayer--Vietoris sequences finish with a possibly non-surjective map between $K_0$ terms. However, the groups $K_n$, $n\le 0$, always satisfy excision, so there is no loss of generality in only working with the non-negative $K$-groups.

Cyclic homology with respect to the base ring $\bb Z$ is denoted $HC_*=HC_*^{\bb Z}$.

\section{Pro-excision in $K$-theory in dimension one}\label{section_pro_excision}
The aim of this foundational section is to describe pro-excision in algebraic $K$-theory, prove that it is satisfied when normalising one-dimensional rings, and state the consequences, namely the long exact sequences of proposition \ref{proposition_standard_consequences} and corollary \ref{corollary_main_application_of_birelative_vanishing}. These consequences are essential for the remainder of the paper.

If $I$ is an ideal of a ring $A$, then $K(A,I)$ is defined to be the homotopy fibre of the map $K(A)\to K(A/I)$; its homotopy groups $K_n(A,I)$ are the {\em relative $K$-groups} associated to $A$ and $I$, and they fit into a long exact sequence \[\cdots\to K_n(A,I)\to K_n(A)\to K_n(A/I)\to\cdots\] If $f:A\to B$ is a morphism of rings carrying $I$ isomorphically to an ideal of $B$, then there is a canonical map $K(A,I)\to K(B,I)$, whose homotopy fibre is denoted $K(A,B,I)$; its homotopy groups are the {\em birelative $K$-groups}, fitting into a long exact sequence \[\cdots\to K_n(A,B,I)\to K_n(A,I)\to K_n(B,I)\to\cdots\]

A non-unital ring $I$ is said to satisfy {\em excision for $K$-theory} if and only if  whenever $I$ is embedded as an ideal into a unital ring $A$, and $f:A\to B$ satisfies the conditions of the previous paragraph, then $K_n(A,I)\to K_n(B,I)$ is an isomorphism for all $n\ge 0$; equivalently, $K_n(A,B,I)=0$ for $n\ge 0$. Informally, the groups $K_n(A,I)$ depends only on $I$, not $A$. Obvious modifications of this terminology will be used. There is a universal choice of such a ring $A$, namely $A=\bb Z\ltimes I$, sometimes denoted $\tilde I$ in the literature.

The following is A.~Suslin's celebrated solution of the excision problem in $K$-theory, building on earlier work of M.~Wodzicki:

\begin{theorem}[Suslin \cite{Suslin1995}]
Let $I$ be a non-unital ring, and set $C=\bb Z$, $\bb Q$, or $\bb Z/m\bb Z$ for any $m\in\bb Z$. Fix $p>0$. Then $I$ satisfies excision for $K_n(-,C)$, for $n\le p$, if and only if $\op{Tor}^{\bb Z\ltimes I}_n(\bb Z,C)=0$ for $n=1,\dots,p$.
\end{theorem}

Suslin--Wodzicki's criterion is rarely satisfied for ideals $I$ occurring in algebraic geometry: it is more appropriate for non-unital $C^*$-algebras and other similar function algebras. Of greater interest to us will be when excision is satisfied as we pass to increasingly fat nilpotent thickenings of an ideal; the following theorem, which we only state in the integral case, is the necessary pro extension of Suslin's result:

\begin{theorem}[T.~Geisser \& L.~Hesselholt {\cite[Thm.~1.1]{GeisserHesselholt2006} \cite[Thm.~3.1]{GeisserHesselholt2011}}]
Let $f:A\to B$ be a morphism of rings, and $I\subset A$ an ideal mapped isomorphically by $f$ to an ideal to $B$. Suppose that the pro abelian group \[\projlimf_r \op{Tor}_n^{\bb Z\ltimes(I^r)}(\bb Z,\bb Z)\] is zero for all $n>0$. Then the pro abelian group \[\projlimf_r K_n(A,B,I^r)\] is zero for all $n\ge 0$, and so $\projlimf_rK_n(A,I^r)\to\projlimf_rK_n(B,I^r)$ is an isomorphism; i.e.~$I$ satisfies `pro-excision' in $K$-theory.
\end{theorem}

\begin{remark}\label{remark_pro_cats}
Before continuing we include a brief discussion about pro abelian groups.

Everything we need about categories of pro objects may be found in one of the standard references, such as the appendix to \cite{ArtinMazur1969}, or \cite{Isaksen2002}. We will use $\op{Pro}Ab$, the category of pro abelian groups. An object of this category is a contravariant functor $X:\cal I\to \cal{C}$, where $\cal I$ is a small cofiltered category (in this paper it is fine to assume $\cal I=\bb N$); this object is denoted $\projlimf_{i\in\cal I} X(i)$. Morphisms are given by the rule \[\Hom_{\op{Pro}Ab}(\projlimf_{i\in\cal I}X,\projlimf_{j\in\cal J}Y):=\projlim_{j\in \cal J}\indlim_{i\in \cal I}\Hom_{Ab}(X(i),Y(j)),\] where the right side is a genuine pro-ind limit in the category of sets, and composition is defined in the obvious way. For example, a pro abelian group $\projlimf_{r\ge 1}X(r)$ is isomorphic to zero if and only if for each $r\ge 1$ there exists $s\ge r$ such that the transition map $X(s)\to X(r)$ is zero.

There is a fully faithful embedding $Ab\to\op{Pro}Ab$ with a right adjoint \[\op{Pro}Ab\to Ab,\quad \projlimf_{i\in\cal I} X(i)\mapsto \projlim_{i\in\cal I} X(i),\tag{\dag}\] which is left exact but not right exact: its first derived functor is precisely $\projlim^1$. Moreover, $\op{Pro}\cal A$ is an abelian category; given a inverse system of exact sequences \[\cdots\To X_{n-1}(i)\To X_n(i)\To X_{n+1}(i)\To\cdots,\] the limit \[\cdots \To \projlimf_{i\in\cal I}X_{n-1}(i)\To \projlimf_{i\in\cal I}X_n(i)\To \projlimf_{i\in\cal I}X_{n+1}(i)\To\cdots\] is an exact sequence in $\op{Pro}Ab$. This does {\em not} imply that \[\cdots \To \projlim_{i\in\cal I}X_{n-1}(i)\To \projlim_{i\in\cal I}X_n(i)\To \projlim_{i\in\cal I}X_{n+1}(i)\To\cdots\] is exact, since (\dag) is not an exact functor.
\end{remark}

Before introducing methods to check the conditions of the Geisser--Hesselholt theorem, we collect together the standard consequences which we will use of the vanishing of the pro birelative $K$-groups $\projlimf_r K_n(A,B,I^r)$:

\begin{proposition}\label{proposition_standard_consequences}
Let $f:A\to B$ be a morphism of rings, and let $I\subset A$ be an ideal mapped isomorphically by $f$ to an ideal of $B$. Suppose that $\projlimf_r K_n(A,B,I^r)=0$ for all $n\ge 0$. Then
\begin{enumerate}
\item There is a natural, long exact, Mayer--Vietoris sequence \[\cdots\to K_n(A)\to\projlimf_r K_n(A/I^r)\oplus K_n(B)\to\projlimf_r K_n(B/I^r)\to\cdots\]
\item Suppose that $J$ (resp.~$J'$) is an ideal of $A$ (resp.~of $B$) containing $I$ (resp.~$f(I)$). Then there is a natural, long exact, Mayer--Vietoris sequence of relative $K$-groups \[\cdots\to K_n(A,J)\to\projlimf_r K_n(A/I^r,J/I^r)\oplus K_n(B,J')\to\projlimf_r K_n(B/I^r,J'/I^r)\to\cdots\]
\item Suppose $J\supseteq I$ is an ideal of $A$ mapped isomorphically to an ideal of $B$. For any $n\ge 0$, the canonical map \[K_n(A,B,J)\to\projlimf_rK_n(A/I^r,B/I^r,J/I^r)\] is an isomorphism.
\item Suppose $J\supseteq I$ is an ideal of $A$ mapped isomorphically to an ideal of $B$. For any $n\ge 0$, the map $K_n(A,B,J)\to K_n(A/I^r,B/I^r,J/I^r)$ is split injective for $r\gg 0$.
\end{enumerate}
\end{proposition}
\begin{proof}
(i)--(iii) follow in a straightforward way by taking the limit over $r$ of exact sequences of homotopy groups, or using pro spectra. (iv) is a consequence of (iii).
\comment{
(i): Consider the following diagram of spectra
\[\xymatrix{
K(A,B,I^r) \ar[r] & K(A,I^r)\ar[r]\ar[d] & K(B,I^r)\ar[d]\\
&K(A)\ar[r]\ar[d] & K(B)\ar[d]\\
&K(A/I^r)\ar[r] & K(B/I^r)
}\]
in which the two columns and top row are homotopy fibre sequences. The vanishing of $\projlimf_r K_n(A,B,I^r)$ for all $n$ implies that $\projlimf_r K_n(A,I^r)\cong \projlimf_r K_n(B,I^r)$, whence the usual splicing argument produces the expected Mayer--Vietoris sequence.

(ii): Consider the following diagram of spectra
\[\xymatrix{
K(A,I^r) \ar[r]\ar[d] & K(A,J) \ar[r]\ar[d] & K(A/I^r,J/I^r) \ar[d] \\
K(A,I^r) \ar[r]\ar[d] & K(A) \ar[r]\ar[d] & K(A/I^r) \ar[d] \\
\ast \ar[r] & K(A/J) \ar[r] & K(A/J)
}\]
in which the bottom two rows and all the columns are homotopy fibre sequences; hence the top row is also a homotopy fibre sequence. The same is true with $B$ and $J'$ in place of $A$ and $J$. Next form the diagram
\[\xymatrix{
K(A,I^r) \ar[r]\ar[d] & K(A,J) \ar[r]\ar[d] & K(A/I^r,J/I^r) \ar[d] \\
K(B,I^r) \ar[r] & K(B,J') \ar[r] & K(B/I^r,J'/I^r) \\
}\]
in which we have just shown that the two rows are homotopy fibre sequences. As we observed in the proof of (i), $\projlimf_r K_n(A,I^r)\cong \projlimf_r K_n(B,I^r)$, whence the Mayer--Vietoris sequence is again obtained by the standard splicing argument.

(iii): Take $J'=f(J)$ in the previous diagram and take homotopy fibres of the vertical arrows; this shows that \[K(A,B,I^r)\to K(A,B,J)\to K(A/I^m,B/I^m,J/I^m)\] is a homotopy fibre sequence, from which the result follows. (iv) is an immediate consequence of (iii).
}
\end{proof}

\begin{remark}
The homotopy groups of $\holim_rK(A,B,I^r)$ fit into short exact sequences \[0\to{\projlim_r}^1K_{n+1}(A,B,I^r)\to\pi_n(\holim_rK(A,B,I^r))\to\projlim_rK_n(A,B,I^r)\to 0\] If $I$ satisfies pro-excision then the two outer terms are zero for all $n\ge1$, and so we deduce that $\holim_rK(A,B,I^r)$ is contractible. Thus \xysquare{K_n(A)}{K_n(B)}{\holim_rK_n(A/I^r)}{\holim_rK_n(B/I^r)}{->}{->}{->}{->} is a homotopy cartesian square of spectra. This leads to variants of the long exact sequences of the previous proposition in which the pro abelian groups are replaced by homotopy groups of homotopy limits of spectra; however, it is much easier to work with pro abelian groups throughout.
\end{remark}

The rest of this section is dedicated to the proof of the following consequence of the Geisser--Hesselholt theorem, implying the vanishing of the pro birelative $K$-groups in a wide variety of situations, the most important of which (for us) we give in the subsequent corollary. We will say that an ideal $I$ of a ring $D$ is {\em locally invertible} if and only if, for every prime ideal $\frak p\subset D$, the ideal $I_\frak p$ is generated by a single non-zero-divisor of $D_\frak p$. If $D$ is Noetherian, then $I$ is invertible if and only if it is locally invertible.

\begin{proposition}\label{proposition_vanishing_of_birelatives_for_locally_free}
Let $f:A\to B$ be a morphism of rings, and $I\subset A$ an ideal mapped isomorphically by $f$ to an ideal to $B$. Suppose moreover that there exists a ring $D$ with the following two properties: $I$ is isomorphic, as a non-unital ring, to a locally invertible ideal of $D$; and $D$ is flat over $\op{Im}(\bb Z\to D)$. Then \[\projlimf_r K_n(A,B,I^r)=0\] for all $n\ge 0$.
\end{proposition}

The following corollary of the proposition will be the case of interest to us. If $A$ is a one-dimensional, Noetherian, reduced ring such that $A\to\tilde A$ is finite, then the {\em conductor ideal} $\frak f:=\Ann_A(\tilde A/A)$ is non-zero; it is the largest ideal of $\tilde A$ contained inside $A$, and the quotients $A/\frak f$ and $\tilde A/\frak f$ are Artinian. The radical of $\frak f$ inside $\tilde A$, i.e.~the ideal $\sqrt{\frak f}=\{b\in\tilde A:b^r\in\frak f\mbox{ for }r\gg0\}$, is equal to the intersection of the finitely many maximal ideals $\frak M$ of $\tilde A$ for which $A_{A\cap\frak M}$ is not normal.

\begin{corollary}\label{corollary_main_application_of_birelative_vanishing}
Let $A$ be a one-dimensional, Noetherian, reduced ring such that $A\to\tilde A$ is finite; let $I$ be a non-zero ideal of $B:=\tilde A$ contained inside $A$, e.g.\ the conductor ideal $\frak f=\Ann_A(B/A)$. Then \[\projlimf_r K_n(A,B,I^r)=0\] for all $n\ge 0$.

In particular, if $J'$ is an ideal of $B$ contained in the radical (taken inside $B$) of $\frak f$, and $J:=A\cap J'$, then there are long exact Mayer--Vietoris sequences
 \[\cdots\to K_n(A)\to\projlimf_r K_n(A/J^r)\oplus K_n(B)\to\projlimf_r K_n(B/J'^r)\to\cdots\]
\[\cdots\to K_n(A,J)\to\projlimf_r K_n(A/J^r,J/J^r)\oplus K_n(B,J')\to\projlimf_r K_n(B/J',J'/J'^r)\to\cdots\]
\end{corollary}
\begin{proof}
To prove the birelative vanishing claim we first reduce to the case where $\Spec A$ is connected and of dimension one. Indeed, the spectrum of $A$ has finitely many components and some of these may be spectra of fields, since we have not insisted $A$ be equi-dimensional. Therefore we may write $A=A'\times\kappa$, and $I=I'\times J$, where $A'$ is an equi-dimensional, Noetherian, reduced ring, $\kappa$ is a finite product of fields, and $I'$ (resp.~$J$) is an ideal of $A'$ (resp.~$\kappa$). Then $B=\tilde{A'}\times\kappa$ and $I'$ is an ideal of $\tilde{A'}$; since $K$-theory is additive for products of rings, we have \[K_n(A,B,I^r)\cong K_n(A',\tilde{A'},I'^r)\oplus K_n(\kappa,\kappa,J^r)\cong K_n(A',\tilde{A'},I'^r).\] In the same way, if $\Spec A'$ has multiple components then we may treat them individually by writing $A'=A''\times A'''$ and proceeding inductively. This reduces us to the connected, one-dimensional case.

 Under this assumption, we will show that the conditions of proposition \ref{proposition_vanishing_of_birelatives_for_locally_free} are satisfied with $D=B$, from which the first part of the corollary follows. Firstly, $B$ is a finite product of Dedekind domains in each of which $I$ has non-zero image, and so $I$ is automatically a locally invertible ideal of $B$. Secondly, the image of $\bb Z$ inside $B$ is either equal to $\bb Z$ (in which case $B$ is a torsion-free, hence flat, $\bb Z$-module), or is equal to a field $k$ since $\Spec A$ is connected (in which case $B$ is automatically flat over $k$).

For the long exact sequences, take $I=J'\frak f$ and just apply proposition \ref{proposition_standard_consequences}, noting that the chains of ideals $\{I^r\}$, $\{J^r\}$ and $\{J'^r\}$ are cofinal in one another.
\end{proof}

\begin{example}\label{example_main_application_of_birelative_vanishing}
If $A$ is a one-dimensional, Noetherian, reduced semi-local ring such that $A\to B:=\tilde A$ is finite, then we may apply corollary \ref{corollary_main_application_of_birelative_vanishing} to the Jacobson radicals $\frak m$, $\frak M$ of $A$, $B$, yielding the long exact Mayer--Vietoris sequences alluded to in the introduction
 \[\cdots\to K_n(A)\to\projlimf_r K_n(A/\frak m^r)\oplus K_n(B)\to\projlimf_r K_n(B/\frak M^r)\to\cdots\]
\[\cdots\to K_n(A,\frak m)\to\projlimf_r K_n(A/\frak m^r,\frak m/\frak m^r)\oplus K_n(B,\frak M)\to\projlimf_r K_n(B/\frak M^r,\frak M/\frak M^r)\to\cdots\]
\end{example}

\begin{remark}
Here we explain some issues surrounding the assumed finiteness of the normalisation morphism $A\to\tilde A$.

Let $A$ be one-dimensional, Noetherian, reduced ring, with normalisation $B=\tilde A$. According to the Krull--Akizuki theorem \cite[Thm.~4.9.2]{Huneke2006}, $B$ is again one-dimensional and Noetherian, hence is a product of finitely many Dedekind domains. Moreover, if $\frak p$ is a prime ideal of $A$, then the prime ideals of $B$ containing $\frak p$ are precisely those finitely many prime ideals occurring in the prime ideal factorisation of $\frak pB$; so only finitely many prime ideals of $B$ contain $\frak p$. In particular, if $A$ is semi-local then so is $B$.

If $A\to B$ is moreover assumed to be a finite morphism, in which case the Krull--Akizuki theorem is of course unnecessary to deduce that $B$ is Noetherian, then the quotient $B/A$ is a finitely generated $A$-module, and so the conductor $\frak f=\Ann_A(B/A)$ is a non-zero ideal of $A$ such that $A/\frak f$ has finite length. Conversely, if $I$ is an ideal of $B$ contained in $A$, and we assume that $I$ is not contained inside any minimal prime ideal of $B$ (to avoid the situation that $B=B'\times B''$ and $I=I'\times\{0\}$), then $I$ contains a non-zero-divisor $f$ of $B$; hence $fB\cong B$ is an ideal of $A$, whence it is finitely generated as an $A$-module.

In conclusion, the finiteness of $A\to B$ is the minimal condition required to formulate corollary \ref{corollary_main_application_of_birelative_vanishing}.
\end{remark}

Now we turn to the proof of proposition \ref{proposition_vanishing_of_birelatives_for_locally_free}. We will check the conditions of the Geisser--Hesselholt theorem by explicitly calculating the Tor groups in question; this is achieved in a standard way using bar resolutions, which we now briefly review.

Let $k$ be a ring, $D$ a $k$-algebra possibly without unit, and let $M,N$ be $D$-modules (recall our convention, for simplicity, that all rings are commutative). The associated two-sided bar construction is the simplicial $D$-module \[\beta_\bullet^k(N,D,M)=N\otimes_k\underbrace{D\otimes_k\cdots\otimes_kD}_{\sub{\sub{$\bullet$ times}}}\otimes_kM,\] with the obvious boundary and degeneracy maps. If $D$ is unital then the presence of the ``extra degeneracy'' (e.g.~\cite[Prop.~1.1.2]{Loday1992}) implies that $\beta_\bullet^k(D,D,M)$ is a resolution of $M$; if $D$ and $M$ are also flat over $k$, it follows that $\beta_\bullet^k(N,D,M)=N\otimes_D\beta_\bullet^k(D,D,M)$ calculates $\op{Tor}_*^D(N,M)$.

Suppose now that $I$ is a non-unital, flat $k$-algebra. Then the normalised chain complex associated to $\beta_\bullet^k(k,k\ltimes I,k)$, which we have just shown calculates $\op{Tor}_*^{k\times I}(k,k)$, can be identified with the subcomplex $\res\beta{}^k_\bullet(I):=\beta_\bullet^{k\ltimes I}(k,I,k)$, where $I$ acts on $k$ as zero. Explicitly, $\res\beta{}^k_\bullet(I)$ is the complex of $k$-modules \[\cdots\To I\otimes_kI\otimes_kI\To I\otimes_kI\stackrel0{\To} k\To0\] where the $k$ sits in degree $0$ and the boundary maps are $b=\sum_{i=0}^{n-1}(-1)^id_i$, where \[d(a_0\otimes\cdots\otimes a_n)=a_0\otimes\cdots\otimes a_ia_{i+1}\otimes\cdots\otimes a_n\qquad (i=0,\dots,n-1).\]

In light of this discussion, the key to proving proposition \ref{proposition_vanishing_of_birelatives_for_locally_free} therefore rests with the following lemma:

\begin{lemma}
Let $k\to D$ be a flat morphism of rings, and $I\subseteq D$ a locally invertible ideal of $D$. Then the canonical map \[H_*(\overline\beta{}^k_\bullet(I^2))\to H_*(\res\beta{}^k_\bullet(I))\] is zero for $*>0$.
\end{lemma}
\begin{proof}
To illustrate the proof we first suppose that $I=tD$ for some non-zero-divisor $t\in D$. For each $n\ge 0$ define \[s:\res\beta{}^k_n(I^2)\to \res\beta{}^k_{n+1}(I),\quad t^2a_0\otimes\cdots\otimes t^2a_n\mapsto t\otimes ta_0\otimes t^2a_1\otimes\cdots\otimes t^2a_n,\] where $a_0,\dots,a_n\in D$. Letting $f:\res\beta{}^k_\bullet(I^2)\to \res\beta{}^k_\bullet(I)$ denote the canonical map induced by the inclusion $I^2\subseteq I$, it is clear that $d_0s=f$ and $d_is=sd_{i-1}$ for $i=1,\dots,n$. So $b's+sb'=f$, whence $f$ induces the zero map on homology.

Now we consider the general case; we must reduce to the situation where have a contracting homotopy by localising enough. First observe that $\res\beta{}^k_\bullet$ has the structure of a complex of left $D$-modules; moreover, for any maximal ideal $\frak m$ of $D$, the flatness of $D\to D_{\frak m}$ implies that $H_*(D_{\frak m}\otimes_D \res\beta{}^k_\bullet(I))=D_{\frak m}\otimes_D H_*(\res\beta{}^k_\bullet(I))$. Moreover, $ID_{\frak m}=tD_{\frak m}$ for some non-zero-divisor $t\in I$, and we may further write $ID_{\frak m}=\indlim_{s\in D\setminus\frak m} \frac{t}{s}D$. So, for $s\in D\setminus\frak m$, let $\cal N_\bullet^s(I)$ denote the subcomplex of $D_{\frak m}\otimes_D \res\beta{}^k_\bullet(I)$ given by \[\cal N_n^s(I)=\left(\frac{t}{s}D\right)\otimes_k I^{\otimes_k n}\] (and similarly $\cal N_\bullet^s(I^2)$, replacing $t$ by $t^2$). Then \[D_{\frak m}\otimes_D H_*(\res\beta{}^k_\bullet(I))=\indlim_{s\in D\setminus\frak m} H_*(\cal N_\bullet^s(I))\] (and similarly for $I^2$), so we have reduced the problem to proving that the canonical map \[f:\cal N_\bullet^s(I^2)\to\cal N_\bullet^s(I)\] induces zero on homology. But this follows similarly to the invertible case, using the contracting homotopy \[s:\cal N_n^s(I^2)\to\cal N_{n+1}^s(I),\quad \frac{t^2}{s}a\otimes b_1\otimes\cdots\otimes b_n\mapsto \frac{t}{s}\otimes ta\otimes b_1\otimes\cdots\otimes b_n,\] where $a\in D$ and $b_1,\dots,b_n\in I^2$.
\end{proof}

\begin{corollary}
Let $k,D,I$ be as in the previous lemma, and let $n>0$. Then \[\op{Tor}_n^{k\ltimes(I^2)}(k,k)\to \op{Tor}_n^{k\ltimes I}(k,k)\] is zero, and so \[\projlimf_r\op{Tor}_n^{k\ltimes(I^r)}(k,k)=0.\]
\end{corollary}
\begin{proof}
The first claim is just a statement of the previous lemma, since we had previously shown $H_*(\res\beta{}^k_\bullet(I))\cong\op{Tor}_*^{k\ltimes I}(k,k)$ (and similarly for $I^2$). The second claim follows by replacing $I$ by successively higher powers of itself and repeatedly applying the first claim, which remains applicable since $I^r$ is still a locally invertible ideal of $D$.
\end{proof}

We need one more lemma, for which we refer to Geisser and Hesselholt's paper for a proof. In fact, their result is much more general and the proof is based on manipulations of Eilenberg--Maclane spectra; for our special case, a direct analysis of the Tor groups using spectral sequences is possible and is due to appear in forthcoming work.

\begin{lemma}
Let $k'\to k$ be a surjective map of rings, and $I$ a non-unital $k$-algebra. Suppose that $\projlimf_r\op{Tor}_n^{k\ltimes(I^r)}(k,k)=0$ for all $n>0$. Then $\projlimf_r\op{Tor}_n^{k'\ltimes(I^r)}(k',k')=0$ for all $n>0$.
\end{lemma}
\begin{proof}
\cite[Prop.~3.6]{GeisserHesselholt2011}.
\end{proof}

Now we may prove that the pro birelative $K$-groups vanish in situations of interest to us:

\begin{proof}[Proof of proposition \ref{proposition_vanishing_of_birelatives_for_locally_free}]
By assumption there is a ring $D$ which contains $I$ as a locally invertible ideal and which is flat over $k:=\op{Im}(\bb Z\to D)$. The previous corollary implies that \[\projlimf_r\op{Tor}_n^{k\ltimes(I^r)}(k,k)=0\] for all $n>0$, and then the previous lemma, with $k'=\bb Z$, implies $\projlimf_r\op{Tor}_n^{\bb Z\ltimes(I^r)}(\bb Z,\bb Z)=0$ for all $n>0$. Finally we apply the Geisser--Hesselholt theorem from the start of the section to complete the proof.
\end{proof}

\section{Applications in equal characteristic}
In this section we apply corollary \ref{corollary_main_application_of_birelative_vanishing} to the study of one-dimensional local rings of equal characteristic. Sections \ref{subsection_finite_char} and \ref{subsection_char_zero} contain preliminary lemmas, and we then establish in section \ref{subsection_main_results} when the pro-excision, Mayer--Vietoris sequences breaks into short exact sequences. In particular, we show that the $K$-theory of a complete, equal charactersitic, one-dimensional local ring is captured by its generic $K$-theory and by all thickenings of its closed point (theorem \ref{theorem_main_theorem_II}), an idea also encapsulated by our application in section \ref{subsection_KH} to homotopy invariant $K$-theory.

Using pro-excision we prove Geller's conjecture in the finite characteristic, perfect residue field case in section \ref{subsubsection_Geller}, and then turn to characteristic zero and a relation with cyclic homology in section \ref{subsection_rel_to_HC}.

Sections \ref{subsection_KH}--\ref{subsection_rel_to_HC} are independent of each other.

\subsection{Finite charactersitic}\label{subsection_finite_char}
We begin with a result in finite characteristic. The following does not seem to be available in exactly this form anywhere in the literature, but readily follows from deep results on the $K$-theory of truncated polynomial algebras by L.~Hesselholt and I.~Madsen, building on older work by S.~Bloch. 

If $A$ is any ring then we write $K^\sub{sym}_n(A)\subseteq K_n(A)$ to denote the symbolic part of the $K$-group, i.e.~the image of the canonical map from Milnor $K$-theory to Quillen $K$-theory; moreover, we write $K_n^\sub{sym}(A[t]/\pid{t^r},\pid t):=\ker(K_n^\sub{sym}(A[t]/\pid{t^r})\to K_n^\sub{sym}(A))$.

\begin{proposition}[Bloch, Hesselholt--Madsen]\label{proposition_symbolic_in_char_p}
Let $A$ be a regular, Noetherian, local $\bb F_p$-algebra, and let $n\ge1$. Then the canonical injection \[\projlimf_rK_n^\sub{sym}(A[t]/\pid{t^r},\pid t)\to \projlimf_rK_n(A[t]/\pid{t^r},\pid t)\] is an isomorphism.
\end{proposition}
\begin{proof}
For a moment let $A$ be any ring, and let $r\ge 1$. We denote by $\bb W_r(A)$ the length $r$ big Witt vectors; this is a commutative, unital ring whose underlying abelian group structure is \[1+tA[t]/1+t^{r+1}A[t].\] Next let $\bb W_r\Omega_A^*$ be the big de Rham Witt complex (e.g., \cite{Hesselholt2001}); it is a differential graded algebra over $\bb W_r(A)$, equipped with a canonical surjection \[\Omega_{\bb W_r(A)}^*\onto \bb W_r\Omega_A^*\] of dg algebras, which is an isomorphism when $*=0$.

Now let $A$ be a regular, Noetherian $\bb F_p$-algebra. In \cite[Thm.~A]{Hesselholt2001}, Hesselholt and Madsen determined the $K$-theory of truncated polynomial algebras over $A$ using R.~McCarthy's \cite{McCarthy1997} comparison theorem between relative $K$-groups and topological cyclic homology: their primary result takes the form of a long exact sequence \[\cdots\To \bigoplus_{i\ge 1}\bb W_i\Omega_A^{n+1-2i} \stackrel{\oplus_iV_r}{\To}\bigoplus_{i\ge 1}\bb W_{ir}\Omega_A^{n+1-2i}\stackrel{\ep}{\To} K_n(A[t]/\pid{t^r},\pid t)\To\cdots,\] where $V_r:\bb W_i\Omega_A^*\to\bb W_{ir}\Omega_A^*$ denotes the Verschiebung map and we will describe a special case of $\ep$ below. To be precise, the result in \cite{Hesselholt2001} is stated only when $A$ is smooth over a perfect field of characteristic $p$, but it was observed in later work that it extended to the more general regular case using Neron--Popescu desingularisation  \cite{Popescu1985, Popescu1986}.

The behaviour of this complex with $r$ was later given in \cite[Thm.~A]{Hesselholt2008}: given $s\ge r$, there is a morphism of complexes
\[\xymatrix{
\cdots \ar[r] & \bigoplus_{i\ge 1}\bb W_i\Omega_A^{n+1-2i} \ar[r]\ar[d]_0 & \bigoplus_{i\ge 1}\bb W_{is}\Omega_A^{n+1-2i}\ar[r]\ar[d] & K_n(A[t]/\pid{t^s},\pid t)\ar[r]\ar[d] & \cdots \\
\cdots \ar[r] & \bigoplus_{i\ge 1}\bb W_i\Omega_A^{n+1-2i} \ar[r] & \bigoplus_{i\ge 1}\bb W_{ir}\Omega_A^{n+1-2i}\ar[r] & K_n(A[t]/\pid{t^r},\pid t)\ar[r] & \cdots
}\]
where the left vertical arrow is zero, the right vertical arrow is the canonical map, and the central vertical arrow is $\bigoplus_{i\ge 1}$ of the maps
\[\xymatrix{
\bb W_{is}\Omega_A^{n+1-2i}\ar[r] &
\bb W_{ir}\Omega_A^{n+1-2i}\ar[r]^{\times \al(s,r)} &
\bb W_{ir}\Omega_A^{n+1-2i},
}\]
where the first map is the canonical reduction map and the second map is multiplication by a certain element $\al(s,r)\in\bb W(\bb F_p)=\projlim_r\bb W_r(\bb F_p)$. Moreover, [Thm.~6.3, op.~cit.] states that for any $r\ge 1$ there exists $s_0\ge r$ such that if $s\ge s_0$ and $i>1$ then $\times \al(s,r)$ is the zero map on $\bb W_{ir}\Omega_A^{n+1-2i}$.

Applying $\projlimf_r$ to the system of complexes, one now quickly sees that the limit of the $i=1$ parts of the $\ep$ maps
\[\projlimf_r\bb W_r\Omega_A^{n-1} \To \projlimf_rK_n(A[t]/\pid{t^r},\pid t)\tag{\dag}\] is an isomorphism, which we will denote $u$. Now, assuming in addition that $A$ is local, we must more carefully describe $u$; we refer the reader to \cite[I.5]{Illusie1979} or \cite[\S8]{Stienstra1985} for more details on what we are about to say.

In \cite[II]{Bloch1977}, Bloch called the relative $K$-group $C_rK_n(A):=K_n(A[t]/\pid{t^{r+1}},\pid t)$ the {\em curves of length $r$ on $K_n$}; let $SC_rK_n(A)=K_n^\sub{sym}(A[t]/\pid{t^{r+1}},\pid t)$ be its symbolic part. Bloch (when $p\neq 2$, an assumption that was removed by \cite[\S2.2]{Kato1980}; see also \cite[\S11]{Stienstra1982}) showed that $S\hat C_rK_{*+1}(A):=\projlimf_r\bigoplus_{n\ge 0}SC_rK_{n+1}(A)$ can be equipped with the structure of pro differential graded associative algebra in such a way that, looking at its degree $0$ component, the inverse of the determinant map \[\{\,\}:\projlimf_r\bb W_r(A)=\projlimf_r1+tA[t]/1+t^{r+1}A[t]\to S\hat C_rK_1(A)\] is an isomorphism of pro rings. The universal property of $\projlimf_r\Omega_{\bb W_r(A)}^*$ as a pro differential graded algebra therefore implies that there is an induced homomorphism of $\projlimf_r\bb W_r(A)$-algebras \[\phi:\projlimf_r\Omega_{\bb W_r(A)}^*\to S\hat C_rK_{*+1}(A)\] (see \cite[Thm.~II.\S6.2.1]{Bloch1977} for a detailed proof). This descends to the quotient $\projlimf_r\bb W_r\Omega_A^*$ by \cite[Thm.~I.5.2]{Illusie1979}, inducing \[u:\projlimf_r\bb W_r\Omega_A^* \to S\hat C_rK_{*+1}(A)= \projlimf_rK_{*+1}^\sub{sym}(A[t]/\pid{t^r},\pid t),\] which is the desired map. Thus the isomorphism (\dag) has image inside $\projlimf_rK_n^\sub{sym}(A[t]/\pid{t^r},\pid t)$, which completes the proof.
\end{proof}

\begin{remark}\label{remark_vanishing_of_even_K_groups_for_truncated_polys}
Suppose that $k$ is a perfect field of characteristic $p\neq0$. Then $\Omega_{\bb W_r(k)}^*=0$ for $*>0$, so the surjection $\Omega_{\bb W_r(k)}^*\onto\bb W_r\Omega_k^*$ shows that $\bb W_r\Omega_k^*=0$ for $*>0$. Applying the isomorphism (\dag) of the previous proof, with $A=k$, we deduce that \[\projlimf_rK_n(k[t]/\pid{t^r},\pid t)=0\] for all $n\ge 2$. In fact, if $n\ge 2$ is even then already $K_n(k[t]/\pid{t^r},\pid t)=0$ for all $r\ge 1$ by \cite[Thm.~A]{Hesselholt1997a}. In our proof of Geller's conjecture we will need the vanishing of this relative group when $n=2$, which can proved in a straightforward classical way by manipulating Steinberg symbols: see e.g., \cite[Lemma 3.4]{Dennis1975}.

Suppose that $A$ is a one-dimensional, Noetherian, reduced semi-local $\bb F_p$-algebra such that $A\to\tilde A$ is finite, and that the residue fields of $A$ are perfect. Then $\tilde A/\frak M^r$ may be identified with a finite product of truncated polynomial rings (see the proof of corollary \ref{corollary_surjectivity_in_limit}) and so the previous paragraph implies that $\projlimf_rK_n(\tilde A/\frak M^r,\frak M/\frak M^r)=0$ for all $n\ge 2$. From corollary \ref{corollary_main_application_of_birelative_vanishing} we deduce that \[K_n(A,\frak m)\cong K_n(A,\frak M)\oplus\projlimf_rK_n(A/\frak m^r,\frak m/\frak m^r)\] for $n\ge 2$. It follows that the map $K_n(A,\frak m)\to K_n(\tilde A,\frak M)$ is surjective and that its kernel is isomorphic to a direct summand of $K_n(A/\frak m^r,\frak m/\frak m^r)$ for $r\gg0$.
\end{remark} 

\subsection{Characteristic zero}\label{subsection_char_zero}
Now we establish the analogue of the previous proposition in characteristic zero. The proof is inspired by ideas from two papers by A.~Krishna \cite{Krishna2005, Krishna2010}, especially section 3 of \cite{Krishna2005}; parts of the proof are also presented in special cases in the author's work \cite{Morrow_Singular_Gersten}.

\begin{remark}
The following lemma and proposition make use of the Hodge decomposition of Hochschild and cyclic homology, which we now briefly review; more details may be found in \cite[\S4.5--4.6]{Loday1992}. Let $k$ be a ring containing $\bb Q$ and let $R$ be a $k$-algebra; all Hochschild and cyclic homologies will be taken with respect to $k$.

The action of the Eulerian idempotents $e_n^{(i)}\in\bb Q[\op{Sym}_n]$, for $1\le i\le n$, on the Hochschild complex $C_\bullet(R)$ are compatible with the boundary maps, thereby resulting in a direct sum decomposition of the Hochschild complex \[C_\bullet(R)=\bigoplus_{i\ge 1}C_\bullet^{(i)}(R),\quad\quad\mbox{where } C_n^{(i)}(R):=e_n^{(i)}C_n(R).\] Thus the Hochschild homology groups decompose as $HH_n(R)=\bigoplus_{i=1}^n HH_n^{(i)}(R)$, where $HH_n^{(i)}(R):=H_n(C_\bullet^{(i)}(R))$. The cyclic homology groups decompose in a similar way: $HC_n(R)=\bigoplus_{i=1}^n HC_n^{(i)}(R)$.

The canonical surjection $HH_n(R)\to\Omega_{R/k}^n$ and anti-symmetrisation map $\ep_n:\Omega_{R/k}^n\to HH_n(R)$ induce isomorphisms \smash{$HH_n^{(n)}(R)\cong \Omega_{R/k}^n$} and \smash{$HC_n^{(n)}(R)\cong\Omega_{R/k}^n/d\Omega_{R/k}^{n-1}$}.
\end{remark}

The following lemma describes the cyclic homology of graded algebras in the limit; I am grateful to C.~Weibel for explaining this style of argument to me.

\begin{lemma}\label{lemma_HC_of_graded_ring}
Let $k$ be a ring (with respect to which all Hochschild/cyclic homologies in the lemma will be taken), and $A=\bigoplus_{w\ge 0}A_w$ a positively graded $k$-algebra. Write $A_{\ge r}=\bigoplus_{w\ge r}A_w\subseteq A$ for each $r\ge 0$. Then for any $n\ge 0$, the canonical map \[HH_n(A)\to\projlimf_r HH_n(A/A_{\ge r})\] is surjective.

Suppose further that $k$ contains $\bb Q$ and that $A$ is a filtered inductive limit of smooth, finite-type $k$-algebras. Then \[\projlimf_r \tilde{HC}_n^{(i)}(A/A_{\ge r})=0\] for all $0\le i<n$, where the notation means the $i^\sub{th}$ piece of the Hodge decomposition of the reduced cyclic homology $\tilde{HC}_*=HC_*/HC_*(A_0)$. Hence \[\projlimf_r\tilde{HC}_n(A/A_{\ge r})\cong\projlimf_r\tilde\Omega^n_{(A/A_{\ge r})/k}/d\tilde\Omega^{n-1}_{(A/A_{\ge r})/k},\] where $\tilde{\Omega}^*$ denotes the quotient by $\Omega_{A_0/k}^*$.
\end{lemma}
\begin{proof}
For any $k$-algebra $R$ we let $C_\bullet(R)=C^k_\bullet(R)=R^{\otimes_k\bullet+1}$ denote its Hochschild complex, whose homology is $HH_*(R)$. If $R$ is positively graded, then $C_\bullet(R)$ breaks into a direct sum of subcomplexes, $C_\bullet(R)=\bigoplus_{w\ge 0}C_\bullet(R)_w$, where the weight $w$ piece is \[C_n(R)_w:=\bigoplus_{i_0+\cdots+i_n=w}R_{i_0}\otimes_k\cdots\otimes_kR_{i_n}\subseteq C_n(R)\] This in turn induces a decomposition of the Hochschild homology, $HH_*(R)=\bigoplus_{w\ge 0}HH_*(R)_w$. Also, write $F^pC_\bullet(R)=\bigoplus_{w\ge p}C_\bullet(R)_w$ for the associated filtration on the complex.

Let $r\ge 0$ be given. By considering two cases we will show that if $s>(n+1)r$ then the map \[\frac{HH_n(A/A_{\ge s})_w}{\op{Im}HH_n(A)_w}\To \frac{HH_n(A/A_{\ge r})_w}{\op{Im}HH_n(A)_w}\tag{\dag}\] is zero for all $w\ge 0$:
\begin{description}
\item[case $\mathbf{w>(n+1)r}$:] Then it is obvious that $C_n(A/A_{\ge r})_w=0$, and so $HH_n(A/A_{\ge r})_w=0$, which clearly suffices to prove our claim.
\item[case $\mathbf{w<s}$:] Notice that $C_n(A,A_{\ge s}):=\ker(C_n(A)\to C_n(A/A_{\ge s}))$ is additively generated by symbols $a_0\otimes\cdots\otimes a_n$ where $|a_i|\ge s$ for at least one $i$, and therefore $C_\bullet(A,A_{\ge s})\subseteq F^sC_\bullet(A)$. Hence $HH_n(A)\to HH_n(A/A_{\ge s})$ is an isomorphism on the weight $w$ pieces whenever $0\le w<s$; so in this case the left side of (\dag) vanishes, which is again sufficient to prove our claim.
\end{description}
Thus for each $n,r\ge 0$ we have found $s\ge r$ such that the map \[\frac{HH_n(A/A_{\ge s})}{\op{Im}HH_n(A)}\To \frac{HH_n(A/A_{\ge r})}{\op{Im}HH_n(A)}\] is zero; this is precisely the statement that $HH_n(A)\to\projlimf_rHH_n(A/A_{\ge r})$ is surjective, thereby proving the first claim of the lemma.

Next, assuming that $k$ contains $\bb Q$ (so that the Hodge decomposition exists) and that $A$ is a filtered inductive limit of smooth, finite-type $k$-algebras (so that the Hochschild--Konstant--Rosenberg theorem \cite[Thm.~3.4.4]{Loday1992} implies $HH_n^{(i)}(A)=0$ for $i\neq n$), we see that \[\projlimf_r HH_n^{(i)}(A/A_{\ge r})=0\] whenever $0\le i<n$. So certainly $\projlimf_r\tilde{HH}_n^{(i)}(A/A_{\ge r})=0$ whenever $0\le i<n$.

Extending this to cyclic homology is straightforward. Indeed, the SBI sequence for the reduced Hochschild and cyclic homology of a graded ring $R$ breaks into short exact sequences \cite[Thm.~4.1.13]{Loday1992}, and the $S$, $B$, and $I$ maps respect the Hodge grading in a suitable way [Prop.~4.6.9, op.~cit.]: \[0\to\tilde{HC}_{n-1}^{(i-1)}(R)\xto{B} \tilde{HH}_n^{(i)}(R)\xto{I}\tilde{HC}_n^{(i)}(R)\to 0.\] Thus we obtain short exact sequences of pro $A$-modules, \[0\to\projlimf_r\tilde{HC}_{n-1}^{(i-1)}(A/A_{\ge r})\to \projlimf_r\tilde{HH}_n^{(i)}(A/A_{\ge r})\to\projlimf_r\tilde{HC}_n^{(i)}(A/A_{\ge r})\to 0,\] and we have just proved that the central term vanishes when $0\le i<n$; hence the right term vanishes, proving the desired vanishing claim for the limit of cyclic homology.

The final claim is immediate from the standard identification of $HC_n^{(n)}$ with $\Omega^n/d\Omega^{n-1}$.
\end{proof}

\begin{proposition}\label{proposition_limit_is_symbols_char_0}
Let $A$ be a regular, Noetherian, local $\bb Q$-algebra, and let $n\ge1$. Then the canonical injection \[\projlimf_rK_n^\sub{sym}(A[t]/\pid{t^r},\pid t)\to \projlimf_rK_n(A[t]/\pid{t^r},\pid t)\] is an isomorphism.
\end{proposition}
\begin{proof}
According to the previous lemma with $k=\bb Q$ and $A[t]$ in place of $A$ (which is a filtered inductive limit of smooth, finite-type $\bb Q$-algebras by Neron--Popescu desingularisation), \[\projlimf_r HC_n^{(i)}(A[t]/\pid{t^r},\pid t)=0\tag{\dag}\] whenever $0\le i<n$.

Notice that $K_n(A[t]/\pid{t^r},\pid t)$ is a relative $K$-group for a nilpotent ideal in a $\bb Q$-algebra, hence is a $\bb Q$-vector space by Weibel \cite[1.5]{Weibel1982}. So T.~Goodwillie's celebrated isomorphism \cite{Goodwillie1986} takes the form \[K_n(A[t]/\pid{t^r},\pid t)\isoto HC_{n-1}(A[t]/\pid{t^r},\pid t).\] This isomorphism respects the Adams/Hodge decompositions by \cite{Cathelineau1990, Cortinas2009}\footnote{
I am grateful to the referee for pointing out that this statement is often not clearly credited. For any ring $R$ there are natural maps \smash{$K_n(R)\xto{ch_n} HN_n(R)\stackrel{B}{\longleftarrow} HC_{n-1}(R)$}, where $ch_n$ is the Chern character and $B$ is the boundary map from cyclic homology to negative cyclic homology.

Now assume that $I$ is a nilpotent ideal of a $\bb Q$-algebra $R$. Then a result of Goodwillie \cite[Thm.~II.5.1]{Goodwillie1985} states that the induced map on relative groups $B:HC_{n-1}(R,I)\to HN_n(R,I)$ is an isomorphism, thereby inducing the relative Chern character $ch_n:K_n(R,I)\to HC_{n-1}(R,I)$. Moreover, Goodwillie \cite{Goodwillie1986} also proved that this relative Chern character is an isomorphism; his proof relied on a ``modified relative Chern character'' $ch_n':K_n(R,I)\to HC_{n-1}(R,I)$, which was only recently proved, by Corti\~nas and Weibel \cite{Cortinas2009a}, to in fact be equal to $ch_n$. However, J.-L.~Cathelineau \cite{Cathelineau1990} had already proved that the modified character $ch'_n$ respects the Adams/Hodge decompositions. In conclusion, $ch_n:K_n(R,I)\to HC_{n-1}(R,I)$ respects the decompositions.
},
 thus inducing \[K_n^{(i)}(A[t]/\pid{t^r},\pid t)\isoto HC_{n-1}^{(i-1)}(A[t]/\pid{t^r},\pid t).\] The vanishing of (\dag) therefore implies that the canonical inclusion \[\projlimf_r K_n^{(n)}(A[t]/\pid{t^r},\pid t)\to \projlimf_r K_n(A[t]/\pid{t^r},\pid t)\] is an isomorphism.

It remains to show that $K_n^{(n)}(A[t]/\pid{t^r},\pid t)$ is entirely symbolic. We start with the following classical Nesterenko--Suslin result \cite[Thm.~4.1]{Suslin1989}: if $R$ is a local ring with an infinite residue field, then $K_n^{(n)}(R)_{\bb Q}\cong K_n^M(R)_{\bb Q}$. Therefore \[K_n^{(n)}(A[t]/\pid{t^r},\pid t)=\ker(K_n^M(A[t]/\pid{t^r})\to K_n^M(A))\otimes_{\bb Z}\bb Q.\] Next we use the following standard result concerning Milnor $K$-theory: If $R$ is a ring and $J\subseteq R$ is an ideal contained inside its Jacobson radical, then $\kappa:=\ker(K_n^M(R)\to K_n^M(R/J))$ is generated by Steinberg symbols of the form $\{a_1,a_2\dots,a_n\}$, where $a_1\in1+J$ and $a_2,\dots,a_n\in\mult A$. (Indeed, if we let $\Lambda$ denote the subgroup of $K_n^M(R)$ generated by such elements, then it is enough to check that \[K_n^M(R/J)\to K_n^M(R)/\Lambda,\quad\{a_1,\dots,a_n\}\mapsto \{\tilde a_1,\dots,\tilde a_n\}\] is well-defined, where $\tilde a\in\mult R$ denotes an arbitrary lift of $a\in\mult{(R/J)}$.) If moreover $J$ is nilpotent and $R\supseteq\bb Q$, then $1+J$ is a divisible group, whence $\kappa$ is also a divisible group and so \[\ker(K_n^M(R)\to K_n^M(R/J))\to\ker(K_n^M(R)\to K_n^M(R/J))\otimes_{\bb Z}\bb Q\] is surjective. Taking $R=A[t]/\pid{t^r}$ and $J=\pid t$ completes the proof.

(An alternative way to show that $K_n^{(n)}(A[t]/\pid{t^r},\pid t)$ is entirely symbolic is to show that $HC_{n-1}^{(n-1)}(A[t]/\pid{t^r},\pid t)$ consists entirely of logarithmic forms: these correspond to symbols in $K$-theory via the Goodwillie isomorphism.)
\end{proof}

\subsection{The short exact Mayer--Vietoris sequences}\label{subsection_main_results}
The propositions of sections \ref{subsection_finite_char} and \ref{subsection_char_zero} yield the following essential corollary:

\begin{corollary}\label{corollary_surjectivity_in_limit}
Let $B$ be a one-dimensional, normal, semi-local ring containing a field; let $\frak M$ denote its Jacobson radical. Then the canonical map \[K_n(B,\frak M)\to \projlimf_rK_n(B/\frak M^r,\frak M/\frak M^r)\] is surjective for all $n\ge0$.
\end{corollary}
\begin{proof}
Each of the groups on the right is zero if $n=0$, so assume $n>0$. Let $\hat B$ be the $\frak M$-adic completion of $B$. There is an isomorphism $\hat B\cong \prod_{\frak n}\hat{B_{\frak n}}$, where $\frak n$ varies over the finitely many maximal ideals of $B$, and each $\hat{B_{\frak n}}$ is a complete discrete valuation ring containing a field, hence isomorphic to $k_{\frak n}[[t]]$ for some field $k_{\frak n}$. Then \[K_n(B/\frak M^r,\frak M/\frak M^r)\cong\bigoplus_{\frak n}K_n(k_{\frak n}[t]/\pid{t^r},\pid t)\] and so the propositions of sections \ref{subsection_finite_char} and \ref{subsection_char_zero} imply that \[\projlimf_rK_n^\sub{sym}(B/\frak M^r,\frak M/\frak M^r)\isoto\projlimf_rK_n(B/\frak M^r,\frak M/\frak M^r).\]

This reduces the claim to showing that $K_n^\sub{sym}(B,\frak M)\to K_n^\sub{sym}(B/\frak M^r,\frak M/\frak M^r)$ is surjective for each $n,r\ge 1$. But this is clear: $\frak M^r$ is contained in the Jacobson radical of $B$, so $\mult B\to\mult{(B/\frak M^r)}$ is surjective.
\end{proof}

We are now prepared to prove our first main result, stating that the long exact, pro-excision, Mayer--Vietoris sequence breaks into short exact sequences in dimension one; special cases were established in \cite[Prop.~3.8]{Morrow_Singular_Gersten} and \cite[Thm.~3.6]{Krishna2005}. This serves as a singular analogue of the Gersten conjecture:

\begin{theorem}\label{theorem_main_theorem}
Let $A$ be a one-dimensional, Noetherian, reduced semi-local ring containing a field, and such that $A\to\tilde A$ is finite; let $\frak m$ and $\frak M$ denote the Jacobson radicals of $A$ and $\tilde A$. For any $n\ge 0$, there is a natural short exact sequence \[0 \to K_n(A,\frak m)\to\projlimf_r K_n(A/\frak m^r,\frak m/\frak m^r)\oplus K_n(\tilde A,\frak M)\to\projlimf_r K_n(\tilde A/\frak M^r,\frak M/\frak M^r)\to0\] of pro abelian groups. In other words, the square \xysquare{K_n(A,\frak m)}{K_n(\tilde A,\frak M)}{\projlimf_r K_n(A/\frak m^r,\frak m/\frak m^r)}{\projlimf_r K_n(\tilde A/\frak M^r,\frak M/\frak M^r)}{->}{->}{->}{->} is bicartesian in Pro$Ab$.
\end{theorem}
\begin{proof}
According to example \ref{example_main_application_of_birelative_vanishing}, there is a long exact Mayer--Vietoris sequence of relative $K$-groups \[\cdots\to K_n(A,\frak m)\to\projlimf_r K_n(A/\frak m^r,\frak m/\frak m^r)\oplus K_n(\tilde A,\frak M)\stackrel{(\ast)}{\to}\projlimf_r K_n(\tilde A/\frak M^r,\frak M/\frak M^r)\to\cdots\] But $B$ satisfies the conditions of the previous corollary, so arrow ($\ast$) is surjective for all $n\ge 0$, which completes the proof.
\end{proof}

\begin{corollary}\label{corollary_to_main_theorem}
Let $A,\frak m,\frak M$ be as in the previous theorem, and fix $n\ge0$.
\begin{enumerate}
\item The kernel (resp.\ cokernel) of the map $K_n(A,\frak m)\to K_n(\tilde A,\frak M)$ is isomorphic to a direct summand of the kernel (resp.\ cokenel) of the map $K_n(A/\frak m^r,\frak m/\frak m^r)\to K_n(\tilde A/\frak M^r,\frak M/\frak M^r)$ for $r\gg0$. In particular, the canonical map \[K_n(A,\frak m)\to K_n(A/\frak m^r,\frak m/\frak m^r)\oplus K_n(\tilde A,\frak M)\] is injective for $r\gg 0$.
\item Let $A'$ also satisfy the conditions of the previous theorem, and suppose that $A\to A'$ is an $S$-analytic isomorphism \cite{Weibel1986}, where $S$ is the set of non-zero-divisors of $A$ (e.g., $A'$ could be an \'etale extension with the same residue field, or the Henselization or completion of $A$). Then the $S$-analytic, relative Mayer--Vietoris sequence breaks into short exact sequences \[0\to K_n(A,\frak m)\to K_n(A',\frak m')\oplus K_n(\tilde A,\frak M)\to K_n(\tilde{A'},\frak M')\to 0\]

In other words, the kernel and cokernel of the map $K_n(A,\frak m)\to K_n(\tilde A,\frak M)$ are unchanged after replacing $A$ by $A'$.
\end{enumerate}
\end{corollary}
\begin{proof}
According to the theorem, the square of groups \xysquare{K_n(A,\frak m)}{K_n(\tilde A,\frak M)}{\projlimf_rK_n(A/\frak m^r,\frak m/\frak m^r)}{\projlimf_r K_n(\tilde A/\frak M^r,\frak M/\frak M^r)}{->}{->}{->}{->} is bicartesian, and so the kernels (resp. cokernels) of the two horizontal arrows are the same. This easily implies (i). For (ii), notice that an $S$-analytic isomorphism $A\to A'$ does not change the bottom row of the diagram, and thus \xysquare{K_n(A,\frak m)}{K_n(\tilde A,\frak M)}{K_n(A',\frak m')}{K_n(\tilde{A'},\frak M')}{->}{->}{->}{->} is also bicartesian.
\end{proof}

The following is the non-relative version of the theorem, in which we suppose that $K_n(\tilde A)\to K_n(\tilde A/\frak M)$ is surjective for all $n\ge 0$; e.g., perhaps $\tilde A\to\tilde A/\frak M$ splits, which is automatic if $A$ is complete but also holds for unibranch, rational singularities.

\begin{theorem}\label{theorem_main_theorem_II}
Let $A,\frak m,\frak M$ be as in the previous theorem, and suppose further that $K_n(\tilde A)\to K_n(\tilde A/\frak M)$ is surjective for all $n\ge0$. Then, for any $n\ge 0$, there is a natural short exact sequence \[0 \to K_n(A)\to\projlimf_r K_n(A/\frak m^r)\oplus K_n(\tilde A)\to\projlimf_r K_n(\tilde A/\frak M^r)\to0.\]
\end{theorem}
\begin{proof}
Using the same argument as the previous theorem, it is enough to show that $K_n(\tilde A)\to\projlimf_rK_n(\tilde A/\frak M^r)$ is surjective. The long exact $K$-theory sequences for $\tilde A\to\tilde A/\frak M$ and $\tilde A/\frak M^r\to\tilde A/\frak M$ split, yielding short exact sequences
\[\xymatrix{
0 \ar[r] & K_n(\tilde A,\frak M) \ar[r]\ar[d] & K_n(\tilde A) \ar[r]\ar[d] & K_n(\tilde A/\frak M)\ar[r]\ar@{=}[d] & 0\\
0 \ar[r] & K_n(\tilde A/\frak M^r,\frak M/\frak M^r) \ar[r] & K_n(\tilde A/\frak M^r) \ar[r] & K_n(\tilde A/\frak M)\ar[r] & 0
}\]
After taking $\projlimf_r$ of the bottom row the left vertical arrow becomes surjective by corollary \ref{corollary_surjectivity_in_limit}, whence the central vertical arrow also becomes surjective.
\end{proof}

The following is the non-relative version of corollary \ref{corollary_to_main_theorem}:

\begin{corollary}
Let $A,\frak m,\frak M$ satisfy the conditions of the previous theorem, and fix $n\ge0$.
\begin{enumerate}
\item The kernel (resp.\ cokernel) of the map $K_n(A)\to K_n(\tilde A)$ is isomorphic to a direct summand of the kernel (resp.\ cokenel) of the map $K_n(A/\frak m^r)\to K_n(\tilde A/\frak M^r)$ for $r\gg0$. In particular, the canonical map \[K_n(A)\to K_n(A/\frak m^r)\oplus K_n(\tilde A)\] is injective for $r\gg 0$.
\item Let $A'$ also satisfy the conditions of the previous theorem, and suppose that $A\to A'$ is an $S$-analytic isomorphism \cite{Weibel1986}, where $S$ is the set of non-zero-divisors of $A$ (e.g., $A'$ could be an \'etale extension with the same residue field, or the Henselization or completion of $A$). Then the $S$-analytic, relative Mayer--Vietoris sequence breaks into short exact sequences \[0\to K_n(A)\to K_n(A')\oplus K_n(\tilde A)\to K_n(\tilde{A'})\to 0\]

In other words, the kernel and cokernel of the map $K_n(A)\to K_n(\tilde A)$ are unchanged after replacing $A$ by $A'$.
\end{enumerate}
\end{corollary}
\begin{proof}
One repeats the proof of corollary \ref{corollary_to_main_theorem}, using absolute rather than relative $K$-groups.
\end{proof}

\begin{remark}
Suppose that $A$ satisfies the conditions of the second of the previous theorems. Then Quillen's proof of the Gersten conjecture in the geometric case, and its extension to the general equal characteristic case by Neron--Popescu desingularisation \cite{Panin2003} (in fact, the case of an arbitrary equal charactersitic discrete valuation ring had already been treated by C.~Sherman \cite{Sherman1978}), tells us that $K_n(\tilde A)\to K_n(F)$ is injective, where $F$ is the total quotient ring of $A$. The corollary therefore implies that $\ker(K_n(A)\to K_n(F))$ is `small enough' to embed into $K_n(A/\frak m^r)$ for some sufficiently large $r$. Thus the philosophy of the second theorem is the following:
\begin{quote}
The $K$-theory of $A$ is determined by its generic information together with all infinitesimal thickenings of its closed point.
\end{quote}
This philosophy will be made even more precise by our result on $KH$-theory in the next section.
\end{remark}

Let $A,\frak m,\frak M$ be as in the theorems; we will finish this section by showing that the conclusions of the second theorem and its corollary are not always valid without some sort of surjectivity/splitting assumption on $\tilde A$. Indeed, supposing that $K_n(A)\to K_n(A/\frak m^r)\oplus K_n(\tilde A)$ is injective for $r\gg 0$, we see from the second long exact sequence of corollary \ref{corollary_main_application_of_birelative_vanishing} that \[\projlimf_r K_{n+1}(A/\frak m^r)\oplus K_{n+1}(\tilde A)\to\projlimf_r K_{n+1}(\tilde A/\frak M^r)\] is surjective. Since $\tilde A\to\tilde A/\frak M$ splits after completion, the map $\projlimf_r K_{n+1}(\tilde A/\frak M^r)\to K_{n+1}(\tilde A/\frak M)$ is also surjective, and so we conclude that \[K_{n+1}(A/\frak m)\oplus K_{n+1}(\tilde A)\to K_{n+1}(\tilde A/\frak M)\] is surjective. Supposing that $A\to A/\frak m$ splits (i.e., $A$ has a coefficient field), the map $K_{n+1}(A/\frak m)\to K_{n+1}(\tilde A/\frak M)$ factors through $K_{n+1}(\tilde A)$, and so we see that in fact \[K_{n+1}(\tilde A)\to K_{n+1}(\tilde A/\frak M)\] is surjective, which of course need not be true. The following provides a specific example:

\begin{example}
Let $A$ be the local ring of the singular point on the nodal curve $Y^2=X^2(X+1)$ over a field $k$. We will show that the map \[K_2(A)\to K_2(A/\frak m^r)\oplus K_2(F)\] is not injective for any $r\ge 0$; to prove this using the above argument we must show that $K_3(\tilde A)\to K_3(\tilde A/\frak M)$ is not surjective:

Well, $B:=\tilde A$ is the semi-local ring obtained by localising $C:=k[t]$ away from two distinct points $x_1,x_2\in \bb A_k^1$. Quillen's localisation theorem implies that there is a short exact sequence \[0\to K_*(k)\to K_*(B)\to \bigoplus_{x\in\bb A_k^1\setminus\{x_1,x_2\}} K_{*-1}(k)\to 0.\] 

However, in the case $*=3$, the boundary map $\bigoplus_{x\neq x_1,x_2}\bor_x:K_3(B)\to\bigoplus_{x\neq x_1,x_2}K_2(k)$ is already surjective when restricted to the symbolic part $K_3^\sub{sym}(B)$ of $K_3(B)$; this is because $K_2(k)$ is generated by symbols and the tame symbols satisfy \[\bor_x\{\theta_1,\theta_2,t_y\}=\begin{cases}\{\theta_1,\theta_2\}&x=y\\0&x\neq y,\end{cases}\] if $x,y\in\bb A_k^1$, $\theta_i\in \mult k$, and $t_y\in k[t]$ is a local parameter at $y$.

Writing $K_3^\sub{ind}=K_3/K_3^\sub{sym}$ as usual, this implies that $K_3^\sub{ind}(k)\to K_3^\sub{ind}(B)$ is surjective. So, if \[K_3^\sub{ind}(B)\to K_3^\sub{ind}(B/\frak M)=K_3^\sub{ind}(k)\oplus K_3^\sub{ind}(k)\] were surjective (which would certainly follow from the surjectivity we are aiming to disprove), we would deduce that the diagonal map $K_3^\sub{ind}(k)\to K_3^\sub{ind}(k)\oplus K_3^\sub{ind}(k)$ were surjective. However, $K_3^\sub{ind}(k)$ is non-zero: for example, its $n$-torsion is $H^0(k,\mu_n^{\otimes 2})$ for any $n$ not divisible by $\op{char}k$ by \cite{Levine1987}, and this is non-zero by picking any such $n$ such that $\mu_n\subseteq\mult k$. This completes the proof.

In fact, with $A$ still being the local ring of the singular point at a node, one can even show that $K_2(A)\to K_2(\hat A)\oplus K_2(\tilde A)$ is not injective. The argument is given in \cite[Prop.~2.14]{Morrow_Singular_Gersten}: it is very similar, using the excision sequence for Weibel's homotopy invariant $K$-theory to reduce to the same non-surjectivity assertion.
\end{example}

\subsection{Application to $KH$-theory}\label{subsection_KH}
Now we turn to applications of the main theorems of the previous section.

For a ring $R$, we denote by $KH(R)$ Weibel's homotopy invariant $K$-theory \cite{Weibel1989a}\comment{(TT's version is Ex.~9.11)}, and we write $K^\sub{sing}(X)$ for the homotopy fibre of $K(X)\to KH(X)$ (there does not appear to be a standardised name for this fibre, but it certainly captures information about the singularities of $R$; if $R$ is $K_0$-regular then $KH(R)\simeq KV(R)$ and so $K^\sub{sing}(R)\simeq K^\sub{nil}(R)$). Recall that $KH$-theory satisfies excision, is invariant under nilpotent extensions, and agrees with $K$-theory on regular rings.

If $A$ is a one-dimensional, Noetherian, reduced, semi-local ring such that $A\to\tilde A$ is finite, then the $S$-analytic isomorphism $A\to\hat A$ yields homotopy cartesian squares in both $K$-theory and $KH$-theory:
\[\xymatrix{
K(A) \ar[r]\ar[d] & K(\Frac A)\ar[d]\\
K(\hat A)\ar[r] & K(\Frac\hat A)
}\quad
\xymatrix{
KH(A) \ar[r]\ar[d] & KH(\Frac A)\ar[d]\\
KH(\hat A)\ar[r] & KH(\Frac\hat A)
}
\]
Taking homotopy fibres shows that $K^\sub{sing}(A)\simeq K^\sub{sing}(\hat A)$; i.e., $K^\sub{sing}$ is an analytic invariant of $A$. The main purpose of this section is to establish corollary \ref{corollary_KH}; first we need a lemma, which we prove in greater generality than required:

\begin{lemma}\label{lemma_K_sing}
Let $A$ be a one-dimensional, Noetherian, reduced ring such that $A\to\tilde A$ is finite, and let $\frak f=\Ann_A(\tilde A/A)\subseteq A$ be the conductor ideal; let $J'$ be a radical ideal of $\tilde A$ contained in $\sqrt{\frak f}$ (the radical of $\frak f$ inside $\tilde A$), and set $J=A\cap J'$. 
\begin{enumerate}
\item Then $KH(A,J)\simeq K(\tilde A,J')$.
\item Moreover, the groups $K_n^\sub{sing}(A)$ fit into a long exact sequence \[\cdots\to K_n^\sub{sing}(A)\to \projlimf_rK_n(A/J^r,J/J^r)\to \projlimf_rK_n(\tilde A/J'^r,J'/J'^r) \to\cdots\]
\end{enumerate}
\end{lemma}
\begin{proof}
$\frak f':=J'\cap \frak f$ is an ideal of both $A$ and $\tilde A$, whose radical in $A$ is $J$ and whose radical in $\tilde A$ is $J'$. We claim that all the following arrows are weak equivalences:
\[KH(A,J)\to KH(A,\frak f')\to KH(\tilde A,\frak f')\to KH(\tilde A,J')\leftarrow K(\tilde A,J')\]
Indeed, the first and third are weak equivalences because $KH$ is nil-invariant; the second is because $KH$ satisfies excision; the fourth is because $\tilde A$ and  $\tilde A/J'$ are regular (the latter is a finite products of fields). This proves (i).

For the second claim, we offer a quick proof using pro spectra, which we have not properly discussed; the cautious reader may replace our homotopy cartesian diagrams of pro spectra by statements about the pro relative groups.

The square of spectra \xysquare{KH(A)}{KH(\tilde A)}{KH(A/J)}{KH(\tilde A/J')}{->}{->}{->}{->} is homotopy cartesian by the proof of the first part. Meanwhile, according to corollary \ref{corollary_main_application_of_birelative_vanishing}, the square \xysquare{K(A)}{K(\tilde A)}{\projlimf_rK(A/J^r)}{\projlimf_rK(\tilde A/J'^r)}{->}{->}{->}{->} is homotopy cartesian; taking homotopy fibres from the second square to the first, we see that  \xysquare{K^\sub{sing}(A)}{\ast}{\projlimf_r\op{hofib}(K(A/J^r)\to KH(A/J))}{\projlimf_r\op{hofib}(K(\tilde A/J'^r)\to KH(\tilde A/J'))}{->}{->}{->}{->} is homotopy cartesian. Since the rings $A/J$ and $\tilde A/J'$ are products of fields, we may replace each $KH$ in the bottom row by $K$, completing the proof.
\end{proof}

\begin{theorem}
Let $A$ be as in the previous lemma, and assume further that it is semi-local, contains a field, and that $K_n(A)\to K_n(A/\frak m)$ is surjective for all $n\ge 1$ (e.g., $A$ local and containing a coefficient field). Then the kernels (resp. cokernels) of the maps \[K_n(A)\to KH_n(A),\quad\quad \projlimf_r K_n(A/\frak m^r,\frak m/\frak m^r)\to\projlimf_r K_n(\tilde A/\frak M^r,\frak M/\frak M^r)\]
are canonically isomorphic for all $n\ge 1$. Here we use our standard notation that $\frak m,\frak M$ are the Jacobson radicals of $A,\tilde A$.
\end{theorem}
\begin{proof}
The surjectivity assumption implies that the long exact sequences for the $K$ and $KH$-theories of $A\to A/\frak m$ break into short exact sequences:
\[\xymatrix{
0 \ar[r] & K_n(A,\frak m) \ar[r]\ar[d] & K_n(A) \ar[r]\ar[d] & K_n(A/\frak m)\ar[r]\ar@{=}[d] & 0\\
0 \ar[r] & KH_n(A, \frak m) \ar[r] & KH_n(A) \ar[r] & KH_n(A/\frak m)\ar[r] & 0
}\]
Thus the left square in the diagram is bicartesian, and so the vertical arrows have the same kernels and cokernels. By the previous lemma we may replace $KH_n(A,\frak m)$ by $K_n(\tilde A,\frak M)$, and then theorem \ref{theorem_main_theorem} completes the proof.
\end{proof}

\begin{corollary}\label{corollary_KH}
Let $A$ satisfy the conditions of the previous theorem. Then, for any $n\ge1$, the map \[K_n(A)\to K_n(A/\frak m^r)\oplus KH_n(A)\] is injective for $r\gg 0$.
\end{corollary}
\begin{proof}
This is an immediate consequence of the previous theorem since the assumption that $K_{n+1}(A)\to K_{n+1}(A/\frak m)$ is surjective implies that $K_n(A/\frak m^r,\frak m/\frak m^r)\subseteq K_n(A/\frak m^r)$.
\end{proof}

\subsection{Geller's conjecture in finite characteristic}\label{subsubsection_Geller}
Now we turn to applications to Geller's conjecture. S.~Geller's conjecture, raised in 1986 \cite{Geller1986}, is the following:
\begin{quote}
``Let $A$ be a one-dimensional, Noetherian, reduced local ring, with total quotient ring $F$. Then the map $K_2(A)\to K_2(F)$ is injective if and only if $A$ is regular.''
\end{quote}
The `if' direction is a classical theorem of K.~Dennis and M.~Stein \cite[Thm.~2.2]{Dennis1975}. Geller herself established the conjecture provided that $A$ has equal characteristic, perfect residue field, and seminormal singularities. Except for this seminormal case there has been no progress on Geller's conjecture in finite characteristic. As discussed in the introduction, Krishna \cite{Krishna2005} reduced Geller's conjecture in characteristic zero to an Artinian analogue, just as Corti\~nas, Geller, and Weibel \cite{Weibel1998} had done earlier for Berger's conjecture on differential forms. 

Using similar ideas we offer the following complete solution to Geller's conjecture for equal characteristic rings with perfect residue field of finite characteristic:

\begin{theorem}[Geller's conjecture in characteristic $p$]\label{theorem_Geller}
Let $A$ be a one-dimensional, Noetherian, reduced, local $\bb F_p$-algebra whose residue field is perfect and such that $A\to\tilde A$ is finite. Suppose that the map \[K_2(A)\To K_2(F)\] is injective, where $F$ is the total quotient ring of $A$. Then $A$ is regular.
\end{theorem}
\begin{proof}
As usual, let $\frak m$ denote the maximal ideal of $A$, and $\frak M$ the Jacobson radical of $\tilde A$. Let $\hat A$ be the $\frak m$-adic completion of $A$; we will show it is sufficient to prove the theorem for $\hat A$ in place of $A$. We first claim that $\hat A$ is reduced and that $\hat A\to\tilde{\hat A}$ is finite. Indeed, by flatness of completion, $\hat A$ embeds into $\hat{\tilde A}=\tilde A\otimes_A\hat A$, which is a finite product of complete discrete valuation rings since $\tilde A$ is a finite product of discrete valuation rings; therefore $\hat A$ is reduced. This shows moreover that $\hat{\tilde A}$ is normal and finite over $\hat A$, whence it is equal to $\tilde{\hat A}$, completing the proof of our claim.

The map $A\to \hat A$ is an $S$-analytic isomorphism \cite{Weibel1986}, where $S$ is the set of non-zero-divisors of $A$, and so there is a resulting long-exact, Mayer--Vietoris sequence \[\cdots \to K_n(A)\to K_n(F)\oplus K_n(\hat A)\to K_n(\hat F)\to \cdots,\] where $\hat F$ is the total quotient ring of $\hat A$. From the assumption that $K_2(A)\to K_2(F)$ is injective, and the fact that $K_1(A)\cong\mult A\to K_1(F)\cong\mult F$ is injective, this long-exact sequence yields a short exact sequence \[0\to K_2(A)\to K_2(F)\oplus K_2(\hat A)\to K_2(\hat F)\to0,\] whence $K_2(\hat A)\to K_2(\hat F)$ is also injective; moreover, $A$ is regular if and only if $\hat A$ is regular, since the two rings have the same dimension and isomorphic cotangent spaces. Therefore we may replace $A$ by $\hat A$ in the rest of the proof. Since any complete, Noetherian, local ring of equal characteristic contains a coefficient field \cite[Thm.~9]{Cohen1946}, and since $\tilde{\hat A}$ is a finite product of such rings, we may therefore henceforth assume that the morphisms $A\to A/\frak m$ and $\tilde A\to\tilde A/\frak M$ split.

Using these splitting assumptions we will now deduce that the map $K_2(A,\frak m)\to K_2(\tilde A,\frak M)$ is injective. Indeed, the splitting assumptions imply that $K_2(A,\frak m)\subseteq K_2(A)$ and $K_2(\tilde A,\frak M)\subseteq K_2(\tilde A)$, so it is enough to prove that $K_2(A)\to K_2(\tilde A)$ is injective; but we are assuming the stronger result that $K_2(A)\to K_2(F)$ is injective.

Moreover, by \cite[Lemma 3.4]{Dennis1975} (see also remark \ref{remark_vanishing_of_even_K_groups_for_truncated_polys}), the pro abelian group $\projlimf_rK_2(k[t]/\pid{t^r},\pid t)$ vanishes if $k$ is any perfect field of finite characteristic. Since $\tilde A/\frak M^r$ is a finite product of such truncated polynomial rings, we deduce that $\projlimf_rK_2(\tilde A/\frak M^r,\frak M/\frak M^r)=0$. The short exact sequence of theorem \ref{theorem_main_theorem} therefore yields a surjection \[K_2(A,\frak m)\to\projlimf_r K_2(A/\frak m^r,\frak m/\frak m^r)\oplus K_2(\tilde A,\frak M)\] (even an isomorphism, but we don't need this). So, considering the maps \[K_2(A,\frak m)\to\projlimf_r K_2(A/\frak m^r,\frak m/\frak m^r)\oplus K_2(\tilde A,\frak M)\xto{\sub{proj}} K_2(\tilde A,\frak M),\] we have proved that the composition is injective and the first arrow is surjective. It follows that \[\projlimf_r K_2(A/\frak m^r,\frak m/\frak m^r)=0,\tag{\dag}\] which is the key to completing the proof.

Each group $K_2(A/\frak m^r)$ is generated by Steinberg symbols, whence the transition maps $K_2(A/\frak m^{r+1})\to K_2(A/\frak m^r)$ are surjective; since $K_2(A/\frak m^r,\frak m/\frak m^r)$ is contained in $K_2(A/\frak m^r)$ for all $r$, thanks to the splitting of $A/\frak m^r\to A/\frak m$, it follows that the transition maps \[K_2(A/\frak m^{r+1},\frak m/\frak m^{r+1})\to K_2(A/\frak m^r,\frak m/\frak m^r)\] are also surjective. Therefore the vanishing of the pro abelian group (\dag) implies the vanishing of the individual groups: $K_2(A/\frak m^r,\frak m/\frak m^r)=0$ for all $r\ge 1$. To complete the proof, it therefore suffices to prove the following claim:
\begin{center}
If $A$ is not regular, then $K_2(A/\frak m^2,\frak m/\frak m^2)$ is non-zero.
\end{center}

Well, if $A$ is not regular, then $\dim_\kappa\frak m/\frak m^2\ge 2$, where $\kappa=A/\frak m$; let $x,y\in\frak m/\frak m^2$ be linearly independent elements. There are various ways to show that the Dennis--Stein symbol \[\langle x,y\rangle\in K_2(A/\frak m^2,\frak m/\frak m^2)\] is non-zero, which will complete the proof, the easiest of which is probably the following. The Maazen--Steinstra presentation of $K_2(A/\frak m^2,\frak m/\frak m^2)$ using Dennis--Stein symbols \cite[Thm.~3.1]{MaazenSteinstra1977} readily implies that there is a well-defined homomorphism  (having fixed a splitting of $A/\frak m^2\to\kappa$) \[K_2(A/\frak m^2,\frak m/\frak m^2)\To\left(\bigwedge\nolimits_\kappa^2A/\frak m^2\right)\bigg/\langle a\wedge b\,:a\mbox{ or } b \mbox{ is in }\kappa\rangle\,\cong\bigwedge\nolimits_\kappa^2\frak m/\frak m^2,\quad \langle a,b\rangle\mapsto a\wedge b,\] where at least one of $a, b\in A/\frak m^2$ belongs to $\frak m/\frak m^2$; this takes $\langle x,y\rangle$ to $x\wedge y\in \bigwedge_\kappa^2\frak m/\frak m^2$, where it is non-zero by choice of $x$ and $y$.
\end{proof}

\comment{
\begin{remark}
A more detailed way to finish the previous proof, when $p>2$, is the following. Firstly, the logarithm isomorphism \cite[Eg.~3.12]{MaazenSteinstra1977}, which exists if $p>2$, \[K_2(A/\frak m^2,\frak m/\frak m^2)\cong \ker(\Omega_{A/\frak m^2}^1\to\Omega_K^1)/d(\frak m/\frak m^2)=\Omega_{A/\frak m^2}^1/d(\frak m/\frak m^2)\] sends $\langle x,y\rangle$ to $y\,dx$. The form $y\,dx$ is non-zero here because there is an isomorphism \cite[Corol.\ 3.2] {Weibel1998}\[\bigwedge\nolimits_K^2\frak m/\frak m^2\isoto \Omega_{A/\frak m^2}^1/d(\frak m/\frak m^2),\quad a\wedge b\mapsto a\,db.\]
Compare with the lemmas of section \ref{subsection_mixed_char_Geller}.
\end{remark}
}

\subsection{Relation to cyclic homology in characteristic zero}\label{subsection_rel_to_HC}
The aim of this section is theorem \ref{theorem_higher_Geller} below, offering an alternative to Geller's conjecture in characteristic zero. The key tool however, which likely has other applications and depends crucially on our main theorem \ref{theorem_main_theorem}, is corollary \ref{corollary_K_vs_HC}, stating that the kernel of a map between $K$-groups is isomorphic to the analogous kernel for cyclic homology.

We start with the cyclic homology version of our main theorem \ref{theorem_main_theorem}:

\begin{proposition}
Let $A$ be a one-dimensional, Noetherian, reduced, semi-local $\bb Q$-algebra such that $A\to\tilde A$ is finite, and let $n\ge 0$. Then the natural square \xysquare{HC_n(A,\frak m)}{HC_n(\tilde A,\frak M)}{\projlimf_r HC_n(A/\frak m^r,\frak m/\frak m^r)}{\projlimf_r HC_n(\tilde A/\frak M^r,\frak M/\frak M^r)}{->}{->}{->}{->} is bicartesian in Pro$Ab$, where all cyclic homologies are taken with respect to $\bb Q$.
\end{proposition}
\begin{proof}
The proof is essentially the same as for $K$-theory, so we will be brief. Firstly,  G.~Corti\~nas' \cite{Cortinas2006} proof of the KABI conjecture implies that we may replace $K_n$ by $HC_{n-1}$ in the vanishing result of corollary \ref{corollary_main_application_of_birelative_vanishing}; hence the subsequent long-exact, Mayer--Vietoris sequences remain valid for cyclic homology in place of $K$-theory. In particular, there is a long exact, Mayer--Vietoris sequence \[\cdots\to HC_n(A,\frak m)\to \projlimf_rHC_n(A/\frak m^r,\frak m/\frak m^r)\oplus HC_n(\tilde A,\frak M)\to\projlimf_rHC_n(\tilde A/\frak M,\frak M/\frak M^r)\to\cdots,\] so to complete the proof it is sufficient to show that \[HC_n(\tilde A,\frak M)\to\projlimf_rHC_n(\tilde A/\frak M,\frak M/\frak M^r)\] is surjective. Fortunately, lemma \ref{lemma_HC_of_graded_ring} implies that the right side of this morphism is \[\projlimf_r\ker\big(\Omega^n_{\tilde A/\frak M^r}/d\Omega^{n-1}_{\tilde A/\frak M^r}\to \Omega^n_{\tilde A/\frak M}/d\Omega^{n-1}_{\tilde A/\frak M}\big)\] Since $\ker(HC_n(\tilde A)\to HC_n(\tilde A/\frak M))$ contains a direct summand isomorphic to $\ker(\Omega^n_{\tilde A}/d\Omega^{n-1}_{\tilde A}\to\Omega^n_{\tilde A/\frak M}/d\Omega^{n-1}_{\tilde A/\frak M})$, one sees that the desired morphism is surjective.
\end{proof}

\begin{corollary}\label{corollary_K_vs_HC}
Let $A$ be as in the previous proposition and let $n\ge 1$. Then the kernel (resp.\ cokernel) of the maps \[K_n(A,\frak m)\to K_n(\tilde A,\frak M),\qquad HC_{n-1}(A,\frak m)\to HC_{n-1}(\tilde A,\frak M)\] are canonically isomorphic.
\end{corollary}
\begin{proof}
This is a consequence of the previous proposition, the analogous theorem for $K$-theory (namely theorem \ref{theorem_main_theorem}), and the Goodwillie isomorphism.
\end{proof}

Before we can use the corollary to prove the main theorem of the section, we need a `cyclic homology criterion for smoothness'. In the following result we denote by $\Omega^\bullet_{R/k}$ the de Rham mixed complex of a $k$-algebra $R$, whose Hochschild and cyclic homologies are respectively $HH_n(\Omega^\bullet_{R/k})=\Omega_{R/k}^n$ and $HC_n(\Omega^\bullet_{R/k})=\Omega^n_{R/k}/d\Omega^{n-1}_{R/k}\oplus\bigoplus_{p\ge1} H_\sub{dR}^{n-2p}(R/k)$.

\begin{lemma}
Let $k\subseteq K$ be an extension of characteristic zero fields, and let $R$ be an essentially finite type $K$-algebra which is not smooth over $K$. Then the canonical map \[HC_n^k(R)\to HC_n(\Omega^\bullet_{R/k})\] is not injective for some $n\ge 2$.
\end{lemma}
\begin{proof}
The Hochschild homology criterion for smoothness \cite{Avramov1992}, which offers an converse to the Hoschschild--Konstant--Rosenberg theorem, says that if $R$ is a finitely generated algebra over a field $K$, then $\Omega_{R/K}^n\to HH_n^K(R)$ is an isomorphism for all $n\ge 0$ if and only if $R$ is smooth over $K$. Obviously this remains valid if $R$ is merely essentially of finite type over $K$, but we will also need to be able to replace $K$ by $k$, which we do as follows.

According to the K\"unneth decomposition and base change for Hochschild homology, one has isomorphisms of graded algebras
\begin{align*}
HH_*^k(R)\otimes_kK\cong HH_*^K(R\otimes_kK)
	&\cong HH_*^K(R)\otimes_KHH_*^K(K\otimes_kK)\\
	&\cong HH_*^K(R)\otimes_k\Omega_{K/k}^*
\end{align*}
Using the same decomposition for de Rham complexes, we see from faithfully flat descent that if $\Omega^*_{R/k}\to HH_*^k(R)$ were to be an isomorphism, then so would be $\Omega^*_{R/K}\to HH_*^K(R)$. In conclusion, since we are assuming $R$ is not smooth over $K$, we deduce that $\Omega^n_{R/k}\to HH_n^k(R)$ is not an isomorphism for some $n\ge 0$.

Next, the following result may be proved by a straightforward induction using the SBI sequences: if $C_\bullet\to D_\bullet$ is a morphism of mixed complexes such that $HH_n(C_\bullet)\to HH_n(D_\bullet)$ is surjective for all $n\ge 0$ and such that $HC_n(C_\bullet)\to HC_n(D_\bullet)$ is injective for all $n\ge 0$, then in fact  $HC_n(C_\bullet)\to HC_n(D_\bullet)$ is an isomorphism for all $n\ge 0$ (and so $HH_n(C_\bullet)\to HH_n(D_\bullet)$ is also an isomorphism for all $n\ge 0$).

So, letting $C_\bullet^k(R)$ denotes the Hochschild complex of $R$ as a $k$-algebra, the usual morphism of mixed complexes \cite[\S2.3]{Loday1992}\[\pi:C_\bullet^k(R)\to\Omega_{R/k}^\bullet,\quad r_0\otimes\cdots\otimes r_n\mapsto r_0\,dr_1\wedge\cdots\wedge r_n\] induces a surjection on the associated Hochschild homologies, but not an isomorphism by what we saw above; therefore the induced map on the cyclic homologies \[HC_n^k(R)\to HC_n(\Omega_{R/k}^\bullet)\] is not injective for some $n\ge 2$ (it is an isomorphism for $n=0,1$), as desired.
\end{proof}

Now we are equipped to prove our higher degree alternative to Geller's conjecture:

\begin{theorem}\label{theorem_higher_Geller}
Let $A$ be a one-dimensional, reduced, semi-local ring which is essentially of finite type over some characteristic zero field, and assume that $A$ is not regular. Then the map \[K_n(A)\to K_n(\Frac A)\] is not injective for some $n\ge3$.
\end{theorem}
\begin{proof}
Just as we replaced $A$ by its completion in the proof of theorem \ref{theorem_Geller}, here we may replace $A$ by a large enough finite extension to ensure that $A\to A/\frak m$ and $\tilde A\to \tilde A/\frak M$ split. Then $K_n(A,\frak m)\subseteq K_n(A)$ and $K_n(\tilde A,\frak M)\subseteq K_n(\tilde A)$, so it is enough to prove that $K_n(A,\frak m)\to K_n(\tilde A,\frak M)$ is not injective for some $n\ge 3$. According to the previous corollary, it is therefore sufficient to prove that $HC_n(A,\frak m)\to HC_n(\tilde A,\frak M)$ is not injective for some $n\ge 2$.

Well, by the previous lemma we may find $n\ge 2$ and non-zero $x\in HC_n(A)$ such that $x$ vanishes in $HC_n(\Omega_{A/\bb Q}^\bullet)$. Since $HC_n(A/\frak m)=HC_n(\Omega_{(A/\frak m)/\bb Q}^\bullet)$ and $HC_n(\tilde A)=HC_n(\Omega_{\tilde A/\bb Q}^\bullet)$, we deduce that $x$ vanishes in both $HC_n(\tilde A)$ and $HC_n(A/\frak m)$, so belongs to $HC_n(A,\frak m)\subseteq HC_n(A)$; this completes the proof.
\end{proof}

\begin{remark}
Using the same techniques, it appears it may be possible to strengthen the conclusion of the previous theorem to `not injective for infinitely many $n\ge 3$'.
\end{remark}

\section{The case of finite residue fields}\label{section_finite_res_field}
In this section we use the pro-excision results from section \ref{section_pro_excision} to study one-dimensional, Noetherian, reduced rings with finite normalisation map, all of whose residue fields are finite: e.g., orders in number fields, affine reduced curves over finite fields, and local versions of these. Notice that corollary \ref{corollary_main_application_of_birelative_vanishing} applies to such rings. Before specialising to the arithmetic setting, we begin by establishing some general results.

We need the following:

\begin{lemma}\label{lemma_thanks_to_Vigleik}
Let $R$ be a finite ring. Then $K_n(R)$ is finite for all $n\ge1$.

In particular, if $A$ is a one-dimensional, Noetherian ring all of whose residue fields are finite, and $I\subseteq A$ is an ideal such that $A/I$ has finite length, then $K_n(A/I^r)$ is finite for all $r,n\ge1$. Hence $\projlim_rK_n(A/I^r)$ is a profinite group.
\end{lemma}
\begin{proof}
I am grateful to V.~Angeltveit for explaining the argument to me. Firstly, Bass stability implies that, for any fixed $n$, $H_n(BGL(R)^+,\bb Z)=H_n(GL(R),\bb Z)=H_n(GL_m(R),\bb Z)$ for $m$ sufficiently large, and $H_n(GL_m(R),\bb Z)$ is finite for $n\ge 1$ since $GL_m(R)$ is a finite group. Thus all the integral homology groups of degree $\ge 1$ of the $K$-theory space $BGL(R)^+$ are finite.

Since $BGL(R)^+$ is an infinite loop space, its $\pi_1$ acts trivially on its $\pi_n$ for all $n\ge 1$, so the theory of Serre classes tells us that \[\pi_n(BGL(R)^+)\mbox{ is finite for all }n\ge 1\Longleftrightarrow H_n(BGL(R)^+,\bb Z)\mbox{ is finite for all }n\ge 1,\] completing the first part of the proof.

The `in particular' claim follows from the fact that $A/I^r$ is finite for all $r\ge 1$.
\end{proof}

\begin{proposition}\label{proposition_finite_kernel_and_cokernel}
Let $A$ be a one-dimensional, Noetherian, reduced ring such that $A\to\tilde A$ is finite, and all of whose residue fields are finite. Then \[K_n(A)\to K_n(\tilde A)\] has finite kernel and cokernel for all $n\ge 1$.

Moreover, if $\ell$ is a prime number invertible in $A/\frak f$, where $\frak f$ denotes the conductor of $A\to \tilde A$, then the kernel and cokernel have no $\ell$-torsion.
\end{proposition} 
\begin{proof}
Consider the following diagram of spectra
\[\xymatrix{
K(A,B,\frak f) \ar[r] & K(A,\frak f)\ar[r]\ar[d] & K(B,\frak f)\ar[d]\\
&K(A)\ar[r]\ar[d] & K(B)\ar[d]\\
&K(A/\frak f)\ar[r] & K(B/\frak f)
}\]
in which the two columns and the top row are homotopy fibre sequences. According to proposition \ref{proposition_standard_consequences}(iv), if $n\ge 1$ then $K_n(A,B,\frak f)$ embeds into $K_n(A/\frak f^r,B/\frak f^r,\frak f/\frak f^r)$ for $r\gg 0$; but this latter group is finite by the previous lemma, and so $K_n(A,B,\frak f)$ is finite. Hence, in the Serre quotient category $Ab/FinAb$ we have $K_n(A,\frak f)\cong K_n(B,\frak f)$ and $K_n(A/\frak f)\cong K_n(B/\frak f)\cong 0$ for all $n\ge1$; the claim follows.

If $\ell$ is a prime number invertible in $A/\frak f$, then the relative groups $K_n(A,\frak f^r)$, $K_n(B,\frak f^r)$ have no $\ell$-torsion by \cite[Consequence 1.4]{Weibel1982}, so we may repeat the previous argument, replacing the category $FinAb$ by the category of finite abelian groups without $\ell$-torsion.
\end{proof}

Our next aim is to show that the long exact, Mayer--Vietoris sequences of pro abelian groups from section \ref{section_pro_excision} can actually be realised to the level to profinite groups; this is a formal consequence of the following lemma:

\begin{lemma}\label{lemma_mittag_leffler}
Let \[\cdots \to A_n\to\projlimf_r A_n(r)\to \projlimf_r B_n(r)\to A_{n-1}\to\cdots\] be a long exact sequence in $\op{Pro}Ab$, where $A_n, A_n(r), B_n(r)\in Ab$. Suppose that the group $B_n(r)$ is finite for all $n,r$. Then the resulting complex of groups \[\cdots\to A_n\to\projlim_r A_n(r)\to \projlim_r B_n(r)\to A_{n-1}\to\cdots\] is exact.
\end{lemma}
\begin{proof}
Firstly, there is no loss of generality in assuming that the long exact sequence in $\op{Pro}Ab$ arises from an inverse system of complexes of abelian groups \[C_\bullet(r)=\qquad \cdots\to A_n\to A_n(r)\to B_n(r)\to A_{n-1}\to \cdots\] For each $n$, the pro abelian group $\projlimf_r\op{Im}(A_n\to A_n(r))$ has surjective transition maps, while $\projlimf_r\op{Im}(A_n(r)\to B_n(r))$ is a limit of finite groups and hence satisfies the Mittag-Leffler condition; it easily follows that $\projlimf_r A_n(r)$ also satisfies the Mittag-Leffler condition.

Thus all terms in the inverse system of complexes $\cdots\to C_\bullet(2)\to C_\bullet(1)$ satisfy the Mittag-Leffler condition, so a standard result on hyper derived functors (e.g., \cite[Thm.~3.5.8]{Weibel1994}) states that there are short exact sequences in homology \[0\to {\projlim_r}^1H_{n+1}(C_\bullet(r))\to H_n(\projlim_rC_\bullet(r))\to\projlim_rH_n(C_\bullet(r))\to 0.\] The two outer groups vanish since the pro abelian groups $\projlimf_r H_n(C_\bullet(r))=H_n(\projlimf_rC_\bullet(r))$ are zero; hence $H_n(\projlim_rC_\bullet(r))$ vanishes for all $n$, as desired.
\comment{
-------------------------------------

The proof is a case of checking that the Mittag-Leffler condition is satisfied in all the necessary places. There is no loss of generality in assuming that the maps $\projlimf_r A_n(r)\to \projlimf_r B_n(r)$ are strict, i.e.~come from a family of compatible maps $A_n(r)\to B_n(r)$.

The boundary map $\projlimf_r B_{n+1}(r)\to A_n$ factors through $B_n(r)$ for $r\gg 0$, hence its image $I_n$ is a finite subgroup of $A_n$. Next put $I_n(r)=\op{Im}(A_n(r)\to B_n(r))$, which is finite since each $B_n(r)$ is finite. This breaks up our long exact sequence into short exact sequences as follows:
\[0\to A_n/I_n \to\projlimf_rA_n(r)\to\projlimf_r I_n(r)\to 0\]
\[0\to \projlimf_rI_n(r)\to\projlimf_rB_n(r)\to I_{n-1}\to 0\]

The realisation functor $\projlim:\op{Pro}Ab\to Ab$ applied to these short exact sequences again gives short exact sequences: in the first case because $\projlim^1$ of a constant system vanishes, and in the second case because $\projlim^1$ of a system of finite abelian groups vanishes. In conclusion we get short exact sequences of groups
\[0\to A_n/I_n \to\projlim_rA_n(r)\to\projlim_r I_n(r)\to 0\]
\[0\to \projlim_rI_n(r)\to\projlim_rB_n(r)\to I_{n-1}\to 0,\]
which may be spliced together to give the desired long exact sequence of groups.}
\end{proof}

\begin{corollary}\label{corollary_group_versions_of_les}
Let $A$ be a one-dimensional, Noetherian, reduced semi-local ring such that $A\to\tilde A$ is finite, and all of whose residue fields are finite; let $\frak m$, $\frak M$ denote the Jacobson radicals of $A$, $\tilde A$. Then there are long exact, Mayer--Vietoris sequences of abelian groups
\[\cdots\to K_n(A)\to\projlim_r K_n(A/\frak m^r)\oplus K_n(B)\to\projlim_r K_n(B/\frak M^r)\to\cdots\]
\[\cdots\to K_n(A,\frak m)\to\projlim_r K_n(A/\frak m^r,\frak m/\frak m^r)\oplus K_n(B,\frak M)\to\projlim_r K_n(B/\frak M^r,\frak M/\frak M^r)\to\cdots\]
\end{corollary}
\begin{proof}
This follows by applying the previous two lemmas to the exact sequences of example \ref{example_main_application_of_birelative_vanishing}.
\end{proof}

In order to prove the analogue in the finite residue field case of our main theorem \ref{theorem_main_theorem_II}, we would like to check condition (i) of the following corollary; the corollary shows that this would follow from showing, informally, that $\projlim_r K_{n+1}(A/\frak m^r)$ is `open' in $\projlim_r K_{n+1}(\tilde A/\frak M^r)$ and that $K_{n+1}(\tilde A)$ is `dense'. Despite being a conceivable `continuity' property of $K$-theory, we will only be able to prove this in special cases.

\begin{corollary}\label{corollary_application_of_group_version_of_les}
Let notation be as in the previous corollary. Then the following are equivalent:
\begin{enumerate}
\item $K_n(A)\to K_n(A/\frak m^r)\oplus K_n(\tilde A)$ is injective for $r\gg0$.
\item $\projlim_r K_{n+1}(A/\frak m^r)\oplus K_{n+1}(\tilde A)\to\projlim_r K_{n+1}(\tilde A/\frak M^r)$ is surjective.
\end{enumerate}
\end{corollary}
\begin{proof}
Since the kernel of $K_n(A)\to K_n(\tilde A)$ is finite by proposition \ref{proposition_finite_kernel_and_cokernel}, one sees that (i) is equivalent to the injectivity of $K_n(A)\to \projlim_rK_n(A/\frak m^r)\oplus K_n(\tilde A)$. The result therefore follows from the first long exact sequence of the previous corollary.
\end{proof}

\subsection{The global case}\label{subsection_global}
Here we apply pro-excision and the preliminary observations of the previous section to deduce finiteness properties of $K$-groups of orders in number fields and non-smooth curves over finite fields.

We begin with the geometric case. A celebrated theorem due to G.~Harder \cite{Harder1977} and C.~Soul\'e \cite{Soule1984} states that the higher $K$-groups of a smooth projective curve over a finite field are finite groups; we can remove the smoothness hypothesis:

\begin{theorem}
Let $k$ be a finite field and $X$ a one-dimensional, reduced scheme, separated and of finite type over $k$. Then $K_n(X)\to K_n(\tilde X)$ has finite kernel and cokernel for all $n\ge 1$. In particular, if $X/k$ is proper then $K_n(X)$ is finite for all $n\ge 1$.
\end{theorem}
\begin{proof}
The first claim is a straightforward induction, using proposition \ref{proposition_finite_kernel_and_cokernel}, on the number of affine patches required to cover $X$. Indeed, let $\pi:\tilde X\to X$ be the normalisation map, and $U,V\subseteq X$ an open cover of $X$ such that the claim has been proved for $U$, $V$, and $U\cap V$. Then Thomason--Trobaugh Zariski descent \cite[Thm.~8.1]{Thomason1990} provides us with long exact Mayer--Vietoris sequences
\[\xymatrix{
\cdots\ar[r] & K_n(X) \ar[r]\ar[d] & K_n(U)\oplus K_n(V) \ar[r]\ar[d] & K_n(U\cap V) \ar[r]\ar[d] & \cdots\\
\cdots\ar[r] & K_n(\tilde X) \ar[r] & K_n(\tilde U)\oplus K_n(\tilde V) \ar[r] & K_n(\tilde U\cap \tilde V) \ar[r] & \cdots
}\]
By hypothesis the right and central vertical arrows have finite kernel and cokernel for $n\ge 1$, whence the same is true of the left vertical arrow.

The second claim follows by applying the Harder--Soul\'e theorem to $\tilde X$.
\end{proof}

The analogue in the arithmetic case of Harder--Soul\'e's result is that of A.~Borel \cite{Borel1974} and D.~Quillen \cite{Quillen1973}, stating that the $K$-groups of the ring of integers of a number field are finitely generated and precisely identifying their ranks.

\begin{theorem}
Let $F$ be a number field with ring of integers $\roi$, and let $A\subseteq \roi$ be an order (i.e. any subring with fractions $F$). Then $K_n(A)$ is a finitely generated abelian group and $K_n(A)\otimes\bb Q\to K_n(\roi)\otimes\bb Q$ is an isomorphism for all $n\ge 1$.
\end{theorem}
\begin{proof}
$A$ automatically satisfies the conditions of proposition \ref{proposition_finite_kernel_and_cokernel} and $\tilde A=\roi$, whence the claims follows from Borel--Quillen.
\end{proof}

\subsection{The local case: finite $\bb Z_p$-algebras}\label{subsection_p_adic_orders}
In this section we will apply pro-excision to the study of certain finite $\bb Z_p$-algebras; to be precise, we will be interested in rings $A$ satisfying the following equivalent conditions:
\begin{enumerate}
\item $A$ is a reduced $\bb Z_p$-algebra which is finitely generated and torsion-free as a $\bb Z_p$-module.
\item $A$ is a one-dimensional, Noetherian, reduced, complete semi-local ring, of mixed characteristic $(0,p)$, and having finite residue fields.
\end{enumerate}
If $A$ satisfies these conditions then its total quotient ring $\Frac A$ is a finite product of finite extensions of $\bb Q_p$, and its normalisation $\tilde A$ is the product of their rings of integers; the Jacobson radicals of $A$ and $\tilde A$ will always be denoted $\frak m$ and $\frak M$ respectively. Note that $A\to\tilde A$ is a finite morphism because $A$, being complete, is excellent and hence has finite normalisation \cite[7.8.3]{EGA_IV_II}.

It would be more straightforward (and intuitive) to assume in addition that $A$ is a local domain, but this restriction would later cause problems.

\begin{example}\label{examples_p_adic_orders}
Here we offer some examples and basic properties of such rings:
\begin{enumerate}
\item Let $\roi$ be the ring of integers of a finite extension of $\bb Q_p$, and let $\frak p$ be the maximal ideal of $\roi$. Then $A:=\bb Z_p+\frak p^s$ satisfies the above conditions for any $s\ge 1$; moreover, $\tilde A=\roi$ and $A/\frak m=\bb F_p$.
\item Let $\roi_1,\dots,\roi_n$ be rings of integers of finite extensions of $\bb Q_p$, and let $\frak p_1,\dots,\frak p_n$ denote their maximal ideals. Let $k$ be a finite field contained in all $\roi_1/\frak p_1,\dots,\roi_n/\frak p_n$; e.g., $k=\bb F_p$ suffices. Then \[A:=\big\{(f_i)\in\prod\nolimits_i\roi_i:f_i\mbox{ mod } {\frak p_i}\mbox{ belongs to $k$ and does not depend on }i\big\}\] is a seminormal local ring satisfying the above conditions, with normalisation $\prod_i\roi_i$ and residue field $k$.
\item Suppose $A$ is local and satisfies the above conditions. Then Hensel's lemma implies that $A$ contains the Teichm\"uller lifts of its residue field $A/\frak m=\bb F_q$. Since $A$ also contains $\bb Z_p$, we deduce it contains $\bb Z_q:=W(\bb F_q)$.
\item If $\roi$ is the ring of integers of a finite extension of $\bb Q_p$ and $G$ is a finite group, then the group algebra $\roi G$ satisfies the above conditions. The only condition which is not immediate is that $\roi G$ is reduced, but it is enough to check this when $G$ is cyclic, in which case $\roi G=\roi[X]/\pid{X^n-1}$ for some $n\ge 1$, and this is reduced since $X^n-1$ has no repeated factors in the UFD $\roi[X]$. For example, if $G$ is the cyclic group of order $p$, then $\bb Z_pG$ is the seminormal ring constructed in (ii) from the fields $\bb Q_p$ and $\bb Q_p(\zeta_p)$; i.e., \[\bb Z_pG\cong\{(f,g)\in\bb Z_p\times\bb Z_p[\zeta_p]:f\mbox{ mod }\pid p=g\mbox{ mod } \pid{1-\zeta}\}.\]
\end{enumerate}
\end{example}

The {\em topological $K$-groups} of a semi-local ring $A$ with Jacobson radical $\frak m$ are defined by \[K_n^\sub{top}(A):=\pi_n(K^\sub{top}(A)),\quad K^\sub{top}(A):=\holim_rK(A/\frak m^r)\] The following isomorphisms, the second of which is a deep theorem of A.~Suslin and A.~Yufryakov, relate completed $K$-groups, topological $K$-groups, and $K$-groups with $\hat{\bb Z}$-coefficients (which we comment on after the lemma); these isomorphisms will be essential tools for our study of the $K$-groups of finite $\bb Z_p$-algebras.

\begin{lemma}\label{lemma_homotopy_description_of_K_groups}
Let $A$ be a $\bb Z_p$-algebra which is finitely generated as a $\bb Z_p$-module; let $\frak m$ be the Jacobson radical of $A$. For all $n\ge1$, there are canonical isomorphisms \[\projlim_r K_n(A/\frak m^r)\cong K_n^\sub{top}(A)\cong K_n(A;\hat{\bb Z}).\]
\end{lemma}
\begin{proof}
The topological $K$-groups fit into short exact sequences \[0\to {\projlim_r}^1 K_{n+1}(A/\frak m^r)\to K_n^\sub{top}(A)\to\projlim_rK_n(A/\frak m^r)\to 0.\] But $K_{n+1}(A/\frak m^r)$ is finite for all $r$ (lemma \ref{lemma_thanks_to_Vigleik}), so the $\projlim^1$ term vanishes and we get isomorphisms $K_n^\sub{top}(A)\cong\projlim_rK_n(A/\frak m^r)$ for $n\ge 0$.

Next, since $A$ is a finite $\bb Z_p$-algebra, A.~Suslin and A.~Yufryakov \cite{Suslin1984a, Suslin1986} proved that the canonical map $K(A)\to K^\sub{top}(A)$ induces a weak equivalence after profinite completion: $K(A)^\comp\stackrel{\sim}{\to} K^\sub{top}(A)^\comp$ (the full argument can be found in the appendix of \cite{Hesselholt1997}). But profinite completion commutes with homotopy limits, and so \[K^\sub{top}(A)^\comp=\holim_r\left( K(A/\frak m^r)^\comp\right)\stackrel{(\ast)}{=}\holim_r K(A/\frak m^r),\] where the final equality follows again from the fact that $K(A/\frak m^r)$ has finite homotopy groups, at least if we ignore $\pi_0$: thus ($\ast$) is actually only an equality if we restrict to a connected component of each side. Hence $\pi_n(K(A)^\comp)=K^\sub{top}_n(A))$ for $n>0$, establishing the second isomorphism.
\end{proof}

\begin{remark}\label{remark_Z_hat_coefficients}
$K$-theory $K_*(-;\hat{\bb Z})=\pi_*(K(-)^\comp)$ with $\hat{\bb Z}$-coefficients is defined to be the homotopy groups of the profinite completion of the $K$-theory spectrum. It is described by short exact sequences
\[0\to\op{Ext}^1_{\bb Z}(\bb Q/\bb Z,K_n(-))\to K_n(-;\hat{\bb Z})\to\Hom_{\bb Z}(\bb Q/\bb Z,K_{n-1}(-))\to 0\]
\[0\to{\projlim_\lambda}^1 K_n(-)[\lambda]\to \op{Ext}^1_{\bb Z}(\bb Q/\bb Z,K_n(-))\to K_n(-)^\comp\to 0\] where $\lambda$ varies over positive integers ordered by divisibility, $C[\lambda]$ denotes the $\lambda$-torsion of an abelian group $C$, and $C^\comp=\projlim_\lambda C/\lambda C$ denotes the $\hat{\bb Z}$-completion of $C$.
\end{remark}

Let $F$ be a finite extension of $\bb Q_p$ with ring of integers $\roi$. We must review the structure of the $K$-groups of $\roi$. For more details we refer the reader to the survey \cite[\S5]{Weibel2005}.

Let $i\ge 1$, and set \[w_i(F)=\#H^0(F,\mu(i)),\quad\quad w_i^{(p)}(F)=\#H^0(F,\mu_{p^\infty}(i)),\] where $\mu$ (resp. $\mu_{p^\infty}$) denote the group of all (resp. $p$-power) roots of unity in $F^\sub{alg}$; then \[\bb Z/w_i(F)\bb Z\cong H^0(F,\mu(i)),\quad\quad \bb Z/w_i^{(p)}(F)\bb Z\cong H^0(F,\mu_{p^\infty}(i)).\] Then $K_{2i}(\roi)$ decomposes into a direct sum \[K_{2i}(\roi)\cong D_i(\roi)\oplus\bb Z/w_i^{(p)}(F)\bb Z,\] where $D_i(\roi)$ is a divisible $\bb Z_{(p)}$-module. On the other hand, the $e$-invariant $e:K_{2i-1}(\roi)\to\bb Z/w_i(F)\bb Z$ induces a direct sum decomposition \[K_{2i-1}(\roi)\cong T_i(\roi)\oplus \bb Z/w_i(F)\bb Z,\] where $T_i(\roi)$ is a torsion-free $\bb Z_{(p)}$-module.

\begin{example}\label{example_K2_groups_of_local_field}
Consider the case $n=2$ as an example. Let $F$ be a finite extension of $\bb Q_p$ and let $\mu_F\subset F$ be the group of roots of unity inside it; put $m=|\mu_F|$. Then the Hilbert symbol induces a surjective homomorphism $H:K_2(F)\to\mu_F$. A theorem of C.~Moore \cite{Moore1968} states that $\ker H=mK_2(F)$ and that this kernel is an uncountable, divisible group (even uniquely-divisible, by \cite{Merkurjev1983}) contained inside $K_2(\roi)$; moreover, $K_2(F)\to\mu_F$ splits.

Restricting to $K_2(\roi)$ one obtains a split short exact sequence \[0\To\ker H=D_1(\roi)\To K_2(\roi)\To\mu^{(p)}_F\cong\bb Z/w_1^{(p)}(F)\bb Z\To0\] where $\mu^{(p)}_F$ denotes the $p$-power roots of unity inside $F$.
\end{example}

Using pro-excision we may generalise these structural descriptions to all finite $\bb Z_p$-algebras satisfying the equivalent conditions (i)--(ii) above:

\begin{theorem}\label{theorem_description_of_K_thy_of_p-adic_orders}
Let $A$ be a reduced $\bb Z_p$-algebra which is finitely generated and torsion-free as a $\bb Z_p$-module, and let $i\ge 1$. Then $K_{2i}(A)$ decomposes as a direct sum \[K_{2i}(A)\cong D_i(A)\oplus W_i^{(p)}(A)\] where $D_i(A)$ is a divisible $\bb Z_{(p)}$-module and $W_i^{(p)}(A)$ is a finite $p$-group. On the other hand, $K_{2i-1}(A)$ decomposes as a direct sum \[K_{2i-1}(A)\cong T_i(A)\oplus W_i(A)\] where $T_i(A)$ is a torsion-free $\bb Z_{(p)}$-module and $W_i(A)$ is a finite group.
\end{theorem}
\begin{proof}
Let $F=\Frac A$ be the total quotient ring of $A$ and let $\roi=\tilde A$ be its normalisation. The claims are clearly true for $\roi$ since it is a finite product of rings of integers of finite extensions of $\bb Q_p$; therefore we may write $D_i(\roi)$, $W_i^{(p)}(\roi)$, etc.\ having the claimed properties.

Proposition \ref{proposition_finite_kernel_and_cokernel} implies that the kernel and cokernel of \[K_{2i}(A)\to K_{2i}(\roi)=D_i(\roi)\oplus W_i^{(p)}(\roi)\] are finite $p$-groups. Therefore the kernel and cokernel of the composition $K_{2i}(A)\to D_i(\roi)$ are finite $p$-groups; but divisible groups have no non-trivial finite images, so this map is actually surjective. The first claim now follows from the algebraic lemma \ref{lemma_on_groups}(ii) which we have postponed until the end of the section to avoid disrupting the exposition.

For the odd case, the same argument implies that the kernel (resp.~cokernel) of the composition $K_{2i-1}(A)\to K_{2i-1}(\roi)\onto T_i(\roi)$ is a finite group (resp.~finite $p$-group). Its image is therefore a torsion-free $\bb Z_{(p)}$-submodule of $T_i(\roi)$, and so we deduce that $K_{2i-1}(A)_\sub{tors}$ is a finite subgroup of $K_{2i-1}(A)$. Finally we use algebraic lemma \ref{lemma_on_groups}(i) to see that \[0\to K_{2i-1}(A)_\sub{tors}\to K_{2i-1}(A)\to K_{2i-1}(A)/K_{2i-1}(A)_\sub{tors}\to 0\] splits.
\end{proof}

\begin{remark}
We stress that although the direct sum decompositions appearing in the proposition are not unique, the summands themselves are. Firstly, $D_i(A)=\bigcap_{n\ge 1}nK_{2i}(A)$ is the maximal divisible subgroup of $K_{2i}(A)$, and $W_i^{(p)}(A)$ is the quotient. Secondly, $W_i(A)$ is the torsion subgroup of $K_{2i-1}(A)$, and $T_i(A)$ is the quotient.
\end{remark}

\begin{remark}
Gabber rigidity \cite{Gabber1992} (and Quillen's calculation of the $K$-theory of finite fields) implies that $W_i(A)\otimes_{\bb Z}\bb Z[\tfrac{1}{p}]\cong K_{2i-1}(A/\frak m)$.
\end{remark}

From the proposition we obtain structural descriptions of the $K$-theory of $A$ with $\hat{\bb Z}$ coefficients, again analogous to what is know for rings of integers of local fields; also, recall from lemma \ref{lemma_homotopy_description_of_K_groups} that $K_n(A;\hat{\bb Z})\cong\projlim_r K_n(A/\frak m^r)$ for all $n\ge 1$:

\begin{corollary}\label{corollary_profinite_K_groups_of_p_adic_order}
Let $A$ be a reduced $\bb Z_p$-algebra which is finitely generated and torsion-free as a $\bb Z_p$-module, let $i\ge 1$, and continue to use the notation introduced in the previous theorem. Then there is a natural isomorphism \[K_{2i}(A;\hat{\bb Z})\cong W_i^{(p)}(A)\] and a short exact sequence \[0\to K_{2i+1}(A)^\comp\to K_{2i+1}(A;\hat{\bb Z})\to\Hom_{\bb Z}(\bb Q/\bb Z,D_i(A))\to 0.\] 
\end{corollary}
\begin{proof}
These readily follow from the standard short exact sequences for $K_*(-;\hat{\bb Z})$ given in remark \ref{remark_Z_hat_coefficients}.
\end{proof}

Now we may prove an arithmetic analogue of the main results in section \ref{subsection_main_results}; unfortunately we can only prove it in odd degrees:

\begin{theorem}
Let $A$ be a reduced $\bb Z_p$-algebra which is finitely generated and torsion-free as a $\bb Z_p$-module, and let $i\ge 1$. Then there is a short exact sequence \[0\to K_{2i-1}(A)\to\projlimf_rK_{2i-1}(A/\frak m^r)\oplus K_{2i-1}(\tilde A)\to\projlimf_rK_{2i-1}(\tilde A/\frak M^r).\] In particular, \[K_{2i-1}(A)\to K_{2i-1}(A/\frak m^r)\oplus K_{2i-1}(\tilde A)\] is injective for all $r\gg 0$.
\end{theorem}
\begin{proof}
Applying lemma \ref{lemma_homotopy_description_of_K_groups} and the previous corollary to $\tilde A$, we see that \[K_{2i}(\tilde A)\to \projlim_rK_{2i}(\tilde A/\frak M^r)=K_{2i}(\tilde A;\hat{\bb Z})=W_i^{(p)}(A)\] is surjective. The claimed injectivity now follows from corollary \ref{corollary_application_of_group_version_of_les}, and then the short exact sequence follows from example \ref{example_main_application_of_birelative_vanishing}.
\end{proof}

We would like to prove the injectivity claim of the previous theorem in even degrees; such a result would imply that the first long exact Mayer--Vietoris sequence of example \ref{example_main_application_of_birelative_vanishing} breaks into short exact sequences, thereby fully extending the main results of section \ref{subsection_main_results} to such finite $\bb Z_p$-algebras. Unfortunately, the best that we can offer in even degree is a list of equivalent conditions which reduces the problem to understanding the torsion in $K_{2i}(A)$, which is unfortunately a difficult problem whose solution is only know when $i=1$:

\begin{proposition}\label{proposition_equivalent_conditions_in_even_case}
Let $A$ be a reduced $\bb Z_p$-algebra which is finitely generated and torsion-free as a $\bb Z_p$-module, and let $i\ge 1$. Then the following are equivalent:
\begin{enumerate}
\item $K_{2i}(A)\to K_{2i}(A/\frak m^r)\oplus K_{2i}(\tilde A)$ is injective for all $r\gg 0$.
\item The canonical map $\Hom_{\bb Z}(\bb Q/\bb Z,D_i(A))\to \Hom_{\bb Z}(\bb Q/\bb Z,D_i(\tilde A))$ is surjective.
\item The canonical map $D_i(A)\to D_i(\tilde A)$ is injective.
\item The canonical map $D_i(A)\to D_i(\tilde A)$ is an isomorphism.
\end{enumerate}
\end{proposition}
\begin{proof}
Corollary \ref{corollary_profinite_K_groups_of_p_adic_order} implies that $W_i^{(p)}(A)=K_{2i}(A)/D_i(A)$ embeds into $K_{2i}(A/\frak m^r)$ for $r\gg 0$, from which (i)$\Leftrightarrow$(iii) easily follows. Next notice that the map $D_i(A)\to D_i(\tilde A)$ has finite kernel and cokernel by proposition \ref{proposition_finite_kernel_and_cokernel}; since the cokernel is divisible it is actually zero, and so this map is surjective. This proves (iii)$\Leftrightarrow$(iv) and gives a short exact sequence \[0\to G\to D_i(A)\to D_i(\tilde A)\to 0\] where $G$ is a finite group. The long exact sequence for $\op{Ext}_{\bb Z}^*(\bb Q/\bb Z,-)$ of this sequence degenerates to \[0\to \Hom_{\bb Z}(\bb Q/\bb Z,D_i(A))\to \Hom_{\bb Z}(\bb Q/\bb Z,D_i(A))\to G\to 0,\] which proves (ii)$\Leftrightarrow$(iii).
\end{proof}

We finish this section by collecting together the various results from abstract algebra which were required in the proof of theorem \ref{theorem_description_of_K_thy_of_p-adic_orders}:

\begin{lemma}\label{lemma_on_groups}
\begin{enumerate}
\item Let $G$ be an abelian group and suppose that the subgroup of torsion elements $G_\sub{tors}$ has finite exponent. Then $G_\sub{tors}$ is a direct summand of $G$.
\item Let \[0\to B\to C\to D\to 0\] be a short exact sequence of abelian groups, where $B$ is finite and $D$ is divisible. Then $C$ is isomorphic to the direct sum of a divisible group and a finite group which is a quotient of $B$.
\end{enumerate}
\end{lemma}
\begin{proof}
(i) is a (perhaps unfamiliar) result from the theory of pure subgroups; e.g., see \cite[4.3.9]{Robinson1996}.

(ii): Since $B$ is finite, its lattice of subgroups $nC\cap B$, $n\ge 1$, is eventually constant, equal to $B_\infty\subseteq B$, say. In other words, there is a fixed integer $m\ge 1$ such that $nmC\cap B=B_\infty$ for all $n\ge 1$. Since $D$ is divisible, the map $mC\to D$ is surjective, and so we obtain an isomorphism $mC/B_\infty\cong D$; moreover, $B_\infty\subseteq\bigcap_{n\ge1}nmC$ by construction, and so it easily follows that $mC$ is a divisible group.

Since $C/B$ is $m$-divisible, the map $B\to C/mC$ is surjective, and thus induces an isomorphism $B/B_\infty\isoto C/mC$. In conclusion we obtain an exact sequence \[0\to mC\to C\to B/B_\infty\to 0,\] where $mC$ is a divisible group and $B/B_\infty$ is a finite group. But Baer's well-known criterion states that divisible groups are injective in the category $Ab$, and so this exact sequence splits.
\comment{
(iii): The assumptions imply that $H_\sub{tors}$ also has bounded exponent, so (i) allows us to write $H=H_\sub{tors}\oplus H_\sub{tf}$, $G=G_\sub{tors}\oplus G_\sub{tf}$, where $H_\sub{tf}$ and $G_\sub{tf}$ are torsion-free. Since $G_\sub{tors}^\comp=G_\sub{tors}$, it is enough to prove the claim for the torsion-free parts. Notice that $H_\sub{tf}\to G_\sub{tf}$ still has kernel and cokernel of bounded exponent, hence is injective and its image contains $nG_\sub{tf}$ for some $n\ge 1$.

Therefore, replacing $G$ by $G_\sub{tf}$, we have reduced the problem to the following: If $G$ is a torsion-free abelian group, then $(nG)^\comp\oplus G\to G^\comp$ is surjective. Well, it easily follows from the torsion-freeness of $G$ that $(nG)^\comp=nG^\comp$ and that $G^\comp/nG^\comp=G/nG$, which completes the proof.
}
\end{proof}

\subsection{Finite $\bb Z_p$-algebras continued: $K_2$ and Geller's conjecture}\label{subsection_mixed_char_Geller}
We continue to study pro-excision for finite $\bb Z_p$-algebras, now focussing on $K_2$ and applications to Geller's conjecture in mixed characteristic. For a reduced $\bb Z_p$-algebra which is finitely generated and torsion-free as a $\bb Z_p$-module, lemma \ref{lemma_homotopy_description_of_K_groups} and corollary \ref{corollary_profinite_K_groups_of_p_adic_order} tell us that \[K_2(A;\hat{\bb Z})=W_1^{(p)}(A)=\projlim_r K_2(A/\frak m^r)=K_2(A/\frak m^r)\quad(r\gg0),\] which is a finite $p$-group. Moreover, example \ref{example_K2_groups_of_local_field} implies that if $A$ is normal then this group is simply the group of $p$-power roots of unity inside $A$.

For $K_2$ of such $\bb Z_p$-algebras we can prove the full analogue of the main theorems of section \ref{subsection_main_results}:

\begin{theorem}
Let $A$ be a reduced $\bb Z_p$-algebra which is finitely generated and torsion-free as a $\bb Z_p$-module. Then $D_1(A)$ is torsion-free and there is a short exact, Mayer--Vietoris sequence \[0\to K_2(A)\to K_2(A;\hat{\bb Z})\oplus K_2(\tilde A)\to K_2(\tilde A;\hat{\bb Z})\to 0.\]
\end{theorem}
\begin{proof}
Example \ref{example_K2_groups_of_local_field} implies that $D_1(\tilde A)$ is not merely divisible, but is also torsion-free. Hence $\Hom_{\bb Z}(\bb Q/\bb Z,D_1(\tilde A))=0$, so condition (ii) of proposition \ref{proposition_equivalent_conditions_in_even_case} implies that $K_2(A)\to K_2(A/\frak m^r)\oplus K_2(\tilde A)$ is injective for $r\gg 0$ and that $D_1(A)\cong D_1(\tilde A)$. Also, $K_1(A)=\mult A\to K_1(\tilde A)=\mult{\tilde A}$ is injective. Combining these two results with the long exact sequence of corollary \ref{corollary_group_versions_of_les} implies that there is a short exact sequence \[0\to K_2(A)\to \projlim_rK_2(A/\frak m^r)\oplus K_2(\tilde A)\to \projlim_rK_2(\tilde A/\frak M^r)\to 0,\] which is the desired result.
\end{proof}

A useful diagrammatic way to restate the theorem is the following:
\[\xymatrix{
0 \ar[r] & D_1(A) \ar[r]\ar[d]^{\cong} & K_2(A) \ar[r]\ar[d] & K_2(A;\hat{\bb Z})\ar[r]\ar[d] & 0\\
0 \ar[r] & D_1(\tilde A) \ar[r] & K_2(\tilde A) \ar[r] & K_2(\tilde A;\hat{\bb Z})\ar[r] & 0
}\]
Thus the right square is bicartesian and all reasonable questions concerning the central vertical arrow can be reduced to an analogous question for the right vertical arrow. In particular we obtain the following, which reduces Geller's conjecture in mixed charactersitic to an Artinian version (c.f.~the opening paragraph of section \ref{subsubsection_Geller}):

\begin{corollary}\label{corollary_Artinian_Geller_mixed_char}
Let $A$ be a one-dimensional, Noetherian, reduced local ring of mixed characteristic $(0,p)$, with finite residue field, and such that $A\to\tilde A$ is finite. Consider the following statements:
\begin{enumerate}
\item $K_2(A)\To K_2(\Frac A)$ is injective.
\item $K_2(\hat A)\To K_2(\Frac \hat A)$ is injective.
\item $K_2(\hat A;\hat{\bb Z})\to K_2(\hat{\tilde A};\hat{\bb Z})$ is injective.
\item $K_2(A/\frak m^r)\To K_2(\tilde A/\frak M^r)$ is injective for $r\gg 0$.
\end{enumerate}
Then (i)$\Rightarrow$(ii)$\Leftrightarrow$(iii)$\Leftrightarrow$(iv).
\end{corollary}
\begin{proof}
(i)$\Rightarrow$(ii) is proved exactly as in the proof of theorem \ref{theorem_Geller}. The remaining equivalences are clear in light of the above bicartesian square since $\hat A$ satisfies the conditions of the previous theorem.
\end{proof}

Using the corollary and a handful of lemmas which we postpone until afterwards, we can now present the first ever results on Geller's conjecture in mixed characteristic. We can rarely show that $A$ is regular, i.e., $\op{embdim}A=1$, but only that $\op{embdim}A\le 2$. This is a consequence of the inability of our methods to detect the crucial element $p\in A$.

Additionally, in case (ii) of the theorem we must exclude one possibility, which we now explain. If $q$ is a power of $p$ and $\roi$ is the ring of integers of a finite extension of $\Frac\bb Z_q$, then \[\{(f,g)\in\bb Z_q\times\roi:f\mbox{ mod }p\bb Z_q=g\mbox{ mod }\frak p\}\tag{\dag}\] ($\frak p$ is the maximal ideal of $\roi$) is a seminormal finite $\bb Z_p$-algebra as in example \ref{examples_p_adic_orders}(ii); if moreover $\roi/\frak p$ is a strict extension of $\bb F_q$ and $\roi$ also contains non-trivial $p$-power roots of unity, then we say that ($\dag$) is {\em bad}.

\begin{theorem}\label{theorem_geller_in_mixed_char}
Let $A$ be a one-dimensional, Noetherian, reduced local ring of mixed characteristic $(0,p)$ such that $A\to\tilde A$ is finite, and with finite residue field. Assume $p\neq2$ and suppose that at least one of the following is true:
\begin{enumerate}
\item $\Frac\hat A$ contains no non-trivial $p$-power roots of unity; or
\item $A$ is seminormal, but $\hat A$ is not isomorphic to a bad ring in the above sense; or
\item $\tilde A$ is local and all $p$-power roots of unity in $\Frac\hat A$ belong to $\hat A$.
\end{enumerate}
If the map $K_2(A)\to K_2(\Frac A)$ is injective then $\op{embdim}A\le 2$. 

In fact, in case (i), if $p\in\frak m^2$ then we actually prove that $\op{embdim}A=1$, i.e.~that $A$ is regular.
\end{theorem}
\begin{proof}
Using corollary \ref{corollary_Artinian_Geller_mixed_char} we see that we may replace $A$ by its completion, which is a reduced, local $\bb Z_p$-algebra which is finitely generated and torsion-free as a $\bb Z_p$-module. Letting $\bb F_q=A/\frak m$, example \ref{examples_p_adic_orders}(iii) says that $\bb Z_q\subseteq A$.

(i): Assume first that $\Frac A$, hence $\tilde A$, contains no non-trivial $p$-power roots of unity. Then $K_2(\tilde A;\hat{\bb Z})=0$, so again using the corollary we deduce that if $K_2(A)\to K_2(\Frac A)$ is injective then $K_2(A/\frak m^r)=0$ for $r\gg 0$; but since $K_2(A/\frak m^r)\to K_2(A/\frak m^2)$ is surjective, this would imply that $K_2(A/\frak m^2)= 0$.

Therefore, according to proposition \ref{proposition_K_2_via_diff_forms} below, $\bigwedge^2_{\bb F_q}\frak m/(\frak m^2+p\bb Z_q)=0$; i.e., \[\dim_{\bb F_q}\frak m/(\frak m^2+p\bb Z_q)\le1,\] from which (i) and the final claim about $p\in\frak m^2$ follow.

(ii): Now assume instead that $A$ is seminormal. By standard theory of seminormal rings (we refer the reader to any of \cite{Davis1978, Weibel1989, Roberts1976, Weibel1980} for such standard theory), $A$ has the following description: if $\frak q_1,\dots,\frak q_m$ are the minimal prime ideals of $A$, and $I_i:=\bigcap_{j\neq i}\frak q_j$, then the maximal ideal of $A$ is $\frak m=I_1+\cdots+I_m$ and this sum is direct.

We first treat the case $m>2$. In this case there clearly exist indices $\al\neq \beta$ and elements $x\in I_\al$, $y\in I_\beta$ such that $x,y,p$ are linearly independent in $\frak m/\frak m^2=I_1/I_1^2\oplus\cdots\oplus I_m/I_m^2$. Then the Dennis--Stein symbol $\did{x,y}\in K_2(A)$ vanishes in $K_2(A/\frak q_i)$ for all $i$, since each $\frak q_i$ contains $x$ or $y$; hence $\did{x,y}$ vanishes in $K_2(\tilde A)$. But $\did{x,y}$ has non-zero image in $K_2(A/\frak m^2)$ by proposition \ref{proposition_K_2_via_diff_forms} and by choice of $x,y$. This shows that $K_2(A)\to K_2(\tilde A)$ cannot be injective when $m>2$.

Next suppose that $m=2$, so that \[A=\{(f_1,f_2)\in\tilde{A/I_1}\times\tilde{A/I_2}:f_1\mbox{ mod } \frak p_1=f_2\mbox{ mod } \frak p_2\in \bb F_q\},\] where $\frak p_i$ denotes the maximal ideal $\tilde{A/I_i}$, which is the ring of integers of a finite extension of $\bb Q_p$. There are two subcases to consider.  Firstly, if the images of $p$ in $I_1/I_1^2$ and $I_2/I_2^2$ span neither of these $\bb F_q$-spaces, then there exist $x\in I_1$ and $y\in I_2$ such that $x,y,p$ are linearly independent in $\frak m/\frak m^2$, and the same proof as in the case $m>2$ works. Secondly suppose that the image of $p$ in $I_2/I_2^2$ spans this space; this case is trickier. Since $I_2/I_2^2$ is the tangent space of the local ring $A/I_1$, we deduce that $A/I_1$ is both regular and unramified, i.e.\ $A/I_1\cong \bb Z_q$. Letting $\roi=A/I_2$, it follows that $A$ is exactly of type (\dag) above; so, by our assumption that $A$ is not bad, either $\roi$ has residue field $\bb F_q$ or $\roi$ contains no non-trivial $p$-power roots of unity. The second case is covered by (i). In the first case it is straightforward to check that $\frak m$ is generated by the elements $(p,0),(0,\pi)$, where $\pi$ is a uniformiser of $\roi$ (c.f.~example \ref{example_geller} below), whence $\op{embdim}A\le 2$. This completes the proof of part (ii) of the theorem.

\comment{
 Let $\zeta$ be a generator of the (possibly trivial) group of $p$-power roots of unity in $\tilde{A/I_1}$; then, in the notation above, $A$ contains the $p$-power root of unity $(\zeta,1)$ and hence contains the $\bb Z_q$-subalgebra $\roi$ it generates, which can be identified with the ring of integers of $\bb Q_q(\zeta)$. Thus the composition \[K_2(\roi;\hat{\bb Z})\to K_2(A;\hat{\bb Z})\to K_2(\tilde{A/I_1};\hat{\bb Z})\] is an isomorphism. Assuming henceforth that $K_2(A)\to K_2(\tilde A)$ is injective implies (using corollary \ref{corollary_Artinian_Geller_mixed_char} and the fact that $K_2(A/I_2;\hat{\bb Z})=0$) that the second arrow in this composition is injective; therefore the second arrow is actually an isomorphism and the first arrow is a split surjection.

Hence $K_2(\roi/\frak n^2)\to K_2(A/\frak m^2)$ is surjective, where $\frak n$ denotes the maximal ideal of $\roi$. Rewriting these $K_2$ groups in terms of differential forms using lemma \ref{lemma_relative_to_absolute_HC} and applying the standard exact sequence for differential forms, we see that $\Omega_{(A/\frak m^2)/(\roi/\frak n^2)}^1/d(A/\frak m^2)=0$. Finally, from a slight modification of lemma \ref{lemma_square_zero}, this can be rewritten as \[\bigwedge\nolimits_{\bb F_q}^2\frak m/(\frak m^2+\pi\roi)=0,\] where $\pi$ is a uniformiser of $\roi$. Just as we finished the proof of part (i), this implies $\op{embdim}A\le 2$.
}

(iii): Finally, suppose that $\tilde A$ is local and all $p$-power roots of unity in $\Frac A$ belong to $A$. Let $\zeta$ be a generator of the (possibly trivial) group of $p$-power roots of unity in $\tilde{A}$ and let $A/\frak m=\bb F_q$. Then $A$ contains $\roi:=\bb Z_q[\zeta]$, which is the ring of integers in $\bb Q_q(\zeta)$. Thus the composition \[\did{\zeta}=K_2(\roi;\hat{\bb Z})\to K_2(A;\hat{\bb Z})\to K_2(\tilde{A};\hat{\bb Z})=\did{\zeta}\] is an isomorphism. Assuming henceforth that $K_2(A)\to K_2(\tilde A)$ is injective implies, using corollary \ref{corollary_Artinian_Geller_mixed_char}, that the second arrow in this composition is injective; therefore the second arrow is actually an isomorphism and the first arrow is a split surjection.

Hence $K_2(\roi/\frak p^2)\to K_2(A/\frak m^2)$ is surjective, where $\frak p$ denotes the maximal ideal of $\roi$. Rewriting these $K_2$ groups in terms of differential forms using lemma \ref{lemma_relative_to_absolute_HC} and applying the standard exact sequence for differential forms, we see that $\Omega_{(A/\frak m^2)/(\roi/\frak p^2)}^1/d(A/\frak m^2)=0$. Finally, from a slight modification of lemma \ref{lemma_square_zero}, this can be rewritten as \[\bigwedge\nolimits_{\bb F_q}^2\frak m/(\frak m^2+\pi\roi)=0,\] where $\pi$ is a uniformiser of $\roi$. Just as we finished the proof of part (i), this implies $\op{embdim}A\le 2$.

\comment{
Firstly, if the image of $p$ in  spans both of these $\bb F_q$-spaces then $\dim_{\bb F_q}\frak m/\frak m^2=2$, as desired.

(iii): Next, assume instead that $A$ is seminormal and that $\tilde A$ is local, with maximal ideal $\frak M$. Then $A$ is a domain and $\tilde A$ is the ring of integers of $F=\Frac A$, which is a finite extension of $\bb Q_p$. Since $A$ is seminormal we have $\frak m=\frak M$ and therefore $A$ is determined by its residue field; so $A=\bb Z_q+\frak M$. If $\zeta$ is a generator of the group of $p$-power roots of unity inside $\tilde A$ then $\zeta=1+(1-\zeta)\in\bb Z_q+\frak M=A$, whence $A$ contains $\roi:=\bb Z_q[\zeta]$. Note that the composition \[K_2(\roi;\hat{\bb Z})\to K_2(A;\hat{\bb Z})\to K_2(\tilde A;\hat{\bb Z})\] is an isomorphism. Our assumption that $K_2(A)\to K_2(\tilde A)$ is injective implies (using corollary \ref{corollary_Artinian_Geller_mixed_char}) that the second arrow in this composition is injective; therefore the second arrow is actually an isomorphism and the first arrow is a (split) surjection.
}
\end{proof}

\begin{example}\label{example_geller}
Let $p>2$ be prime and let $q=p^l$ be a power of $p$. Then $A:=\bb Z_p+p^s\bb Z_q$ is a reduced, local $\bb Z_p$-algebra which is finitely generated and torsion-free as a $\bb Z_p$-module; $A$ has normalisation $\bb Z_q$, maximal ideal $\frak m=p\bb Z_p+p^s\bb Z_q$, residue field $\bb F_p$, and embedding dimension $\dim_{\bb F_p}\frak m/\frak m^2=l+1$. Indeed, an $\bb F_p$ basis for $\frak m/\frak m^2$ is given by $p$, $p^s\theta_i$, $i=1,\dots,l$, where $\{\theta_i\}$ are Teichm\"uller lifts of a basis of $\bb F_q$ as a $\bb F_p$-space.

So, assuming that $l\neq 1$, part (i) of the theorem implies that $K_2(A)\to K_2(\bb Z_q)$ is not injective. However, if $l=1$ then $\dim_{\bb F_p}\frak m/(p\bb Z_p+\frak m^2)=1$ and so $K_2(A/\frak m^2)=0$, telling us nothing about the putative injectivity of $K_2(A)\to K_2(\bb Z_q)$. It seems likely that the Dennis--Stein symbol $\did{p,p^s\theta_1}\in K_2(A/\frak m^r)$ will be non-zero for $r\gg0$, which would prove non-injectivity of $K_2(A)\to K_2(\bb Z_q)$, but I cannot prove it.
\end{example}

The theorem required various explicit descriptions of $K_2(A/\frak m^r)$, especially when $r=2$, which we establish in the remainder of this section by modifying classical results such as those in \cite{MaazenSteinstra1977} and \cite{Weibel1980a}:

\begin{lemma}\label{lemma_generalisation_of_MS}
Let $R$ be a ring containing a nilpotent ideal $I$; let $N$ be the smallest integer for which $I^N=0$, and assume that $N!\in\mult R$. Then there is a natural isomorphism \[K_2(R,I)\cong HC_1(R,I).\]
\end{lemma}
\begin{proof}
If $N=1$ then both sides vanish; assume henceforth that $N>1$. So, in particular, $2$ is invertible in $R$, which implies that $HC_0, HC_1$ and $HC_2$ may be defined using Connes' complex rather than the cyclic bicomplex: see the remark in \cite[\S2.1]{Loday1992}. So the relative group $HC_1(R,I)$ admits the following description: First let $C_1(R,I)$ be the submodule of $R\otimes_\bb ZR$ generated by symbols $a\otimes b$ where at least one of $a,b$ lies in $I$; then $HC_1(R,I)$ is the abelian group obtained by quotienting $C_1(R,I)$ by the relations
\begin{align*}
&ab\otimes c-a\otimes bc + ca\otimes b=0\quad\quad (a,b,c\in R,\mbox{ at least one in }I)\\
&a\otimes b+b\otimes a=0\quad\quad (a,b\in R,\mbox{ at least one in }I)
\end{align*}

On the other hand, F.~Keune \cite[Thm.~15]{Keune1978} proved that $K_2(R,I)$ admits the following description by Dennis--Stein symbols: It is the abelian group generated by symbols $\did{a,b}$, where $a,b\in R$ and at least one of $a,b$ lies in $I$, modulo the relations
\begin{align*}
&\did{a,b}=-\did{-b,-a}\\
&\did{a,b}+\did{a,c}=\did{a,b+c+abc}\\
&\did{a,bc}=\did{ab,c}+\did{ac,b}
\end{align*}

In the case when $R\to R/I$ is split, and with the same hypothesis that $N!\in\mult R$, H.~Maazen and J.~Stienstra \cite[E.g.~3.12]{MaazenSteinstra1977} explicitly constructed an isomorphism \[K_2(R,I)\cong \ker(\Omega_R^1\to\Omega_{R/I}^1)/dI=HC_1(R,I),\quad \did{a,b}\mapsto l(a,b)\,da,\] where $l(X,Y)$ is a formal logarithm function. Their proof works verbatim in the general situation when $R\to R/I$ is not necessarily split, replacing $l(a,b)\,da$ by $l(a,b)\otimes a\in HC_1(R,I)$.
\end{proof}

Next we pass from the relative groups to the absolute ones:

\begin{lemma}\label{lemma_relative_to_absolute_HC}
Let $R$ be a finite ring with Jacobson radical $\frak M$, and suppose that $R/\frak M$ is a finite product of finite fields of characteristic $p$. Assume that $\frak M^{p-1}=0$. Then there is a natural isomorphism \[K_2(R)\cong \Omega_R^1/dR\]
\end{lemma}
\begin{proof}
The relative group $K_2(R,\frak M)$ is a $\bb Z_{(p)}$-module, while $K_3(R/\frak M)$ is a finite group of order prime to $p$ (thanks to Quillen's calculation of the $K$-theory of finite fields); therefore the map $K_3(R/\frak M)\to K_2(R,\frak M)$ is zero. Moreover, $K_2(R/\frak M)=0$, again because $R/\frak M$ is a finite product of finite fields, and so we have proved that $K_2(R,\frak M)\cong K_2(R)$.

Now we will prove the analogous result for cyclic homology, which will finish the proof (using the previous lemma). Notice that it does not matter whether we compute $HH$ and $HC$ with respect to $\bb Z$ or with respect to the image of $\bb Z$ inside the ring under question; we will freely pass between the two without indictation. Since $R/\frak M$ is a finite product of finite fields of characteristic $p$,  it is smooth over $\bb F_p$ and $\Omega_{R/\frak M}^*=0$ for $*\ge1$. The Hochschild--Kostant--Rosenberg theorem \cite[Thm.~3.4.4]{Loday1992} and SBI sequence therefore implies that $HC_2(R/\frak M)=HC_0(R/\frak M)=R/\frak M$. The existence of Teichm\"uller lifts easily implies that $HC_0(R)\to HC_0(R/\frak M)$ is surjective, and therefore $HC_2(R)\to HC_2(R/\frak M)$ is surjective. Moreover, $HC_1(R/\frak M)=0$, completing the proof that $HC_1(R,\frak M)=HC_1(R)=\Omega_R^1/dR$.
\end{proof}

Next we specialise to the case of a square-zero ideal:

\begin{lemma}\label{lemma_square_zero}
Let $R$ be a ring containing an ideal $I$ such that $I^2=0$ and such that $2\in\mult R$. Assume that the composition $H_\sub{dR}^0(R)\to R\to R/I$ is surjective, and let $k$ be any subring of $H_\sub{dR}^0(R)$ which surjects onto $R/I$. Then there is a natural isomorphism of $k$-modules \[\bigwedge\nolimits_{R/I}^2\res I\cong\Omega_R^1/dR,\quad a\wedge b\mapsto a\,db,\] where $\res I:=I/I\cap k$.
\end{lemma}
\begin{proof}
First notice that $R=k+I$, though not necessarily as a direct sum, and that $\Omega_R^1=\Omega_{R/k}^1$, so we may work with $\Omega_{R/k}^1$ throughout; we will identify $\Omega_{R/k}^1/dR$ with $HC_1^k(R)$ via $a\,db\leftrightarrow a\otimes b$, which has the following presentation: it is the quotient of $R\otimes_kR$ by the $k$-submodule generated by
\begin{align*}
&ab\otimes c-a\otimes bc + ca\otimes b=0\quad\quad (a,b,c\in R),\\
&a\otimes b+b\otimes a=0\quad\quad (a,b\in R).
\end{align*}

Let $\Lambda$ be the $k$-submodule of $\bigwedge_k^2R$ generated by terms $a\wedge b$ where at least one of $a,b$ belongs to $k$. We claim that there is an isomorphism \[HC_1^k(R)\cong \big(\bigwedge\nolimits_k^2R\big)/\Lambda,\quad a\otimes b\leftrightarrow a\wedge b.\] It is clear that $(\bigwedge_k^2R)/\Lambda\to\Omega_{R/k}^1/dR$, $a\wedge b\mapsto a\,db$ is well-defined, thereby defining the isomorphism in one direction. In the other direction, it is evident that $a\otimes b\mapsto a\wedge b\mod\Lambda$ sends $a\otimes b+b\otimes a$ to zero, so it remains only to check that \[ab\wedge c-a\wedge bc+ca\wedge b=0\mod\Lambda\] for all $a,b,c\in R$. Since the identity is linear and symmetric in $a,b,c$ it is sufficient to prove it in the following two cases:
\begin{enumerate}
\item $a\in k$: Then the identity becomes \[ab\wedge c-a\wedge bc-ab\wedge c=-a\wedge bc\equiv0\mod\Lambda.\]
\item $a,b,c\in I$: Then the identity vanishes since $I^2=0$.
\end{enumerate}
This proves that $HC_1^k(R)\to (\bigwedge_k^2R)/\Lambda$ is well-defined, completing the proof of our claimed isomorphism.

Finally, it is straightforward to see that the surjection $\bigwedge_k^2I\to(\bigwedge_k^2R)/\Lambda$ descends to an isomorphism $\bigwedge_{k/k\cap I}^2\res I\cong(\bigwedge_k^2R)/\Lambda$.
\end{proof}

Now we reach the main application of the lemmas; recall from example \ref{examples_p_adic_orders}(iii) that if a finite local $\bb Z_p$-algebra has residue field $\bb F_q$ then it contains $\bb Z_q$.

\begin{proposition}\label{proposition_K_2_via_diff_forms}
Let $A$ be a reduced local $\bb Z_p$-algebra which is finitely generated and torsion-free as a $\bb Z_p$-module, with residue field $\bb F_q$, and assume $p>2$. Then there is a natural isomorphism \[K_2(A/\frak m^2)\cong\bigwedge\nolimits^2_{\bb F_q}\frak m/\frak m',\quad\did{x,y}\leftrightarrow y\wedge x\] where $\frak m'=\frak m^2+p\bb Z_q$.
\end{proposition}
\begin{proof}
Combine the previous three lemmas, with $R=A/\frak m^2$ and $k=\bb Z_q/\bb Z_q\cap\frak m^2$.
\end{proof}

\vspace{1cm}\noindent Matthew Morrow,\\
University of Chicago,\\
5734 S. University Ave.,\\
Chicago,\\
IL, 60637,\\
USA\\
{\tt mmorrow@math.uchicago.edu}\\
\url{http://math.uchicago.edu/~mmorrow/}\\
\end{document}